\documentclass[smallextended]{svjour3}       % onecolumn (second format)

\usepackage{amsmath}
\usepackage{amstext}
\usepackage{amsbsy}
\usepackage{amsopn}
\usepackage{amscd,amsxtra,upref,amssymb}
\usepackage[dvips]{graphicx}
\usepackage{color}
\usepackage{algorithmic}
\usepackage{algorithm,cite,soul}

\usepackage[colorinlistoftodos]{todonotes}

\newcommand{\Real}{\mathbb R}
\newcommand{\Cplx}{\mathbb C}

\newcommand{\diag}{\operatorname{diag}}
\newcommand{\bfe}{\mathbf{e}}

\newcommand{\Imunit}{\mathfrak{i}}

\newcommand{\mb}[1]{\left(\begin{array}{#1}}
\newcommand{\me}{\end{array}\right)}
\newcommand{\eqb}{\begin{equation}}
\newcommand{\eqe}{\end{equation}}
\newcommand{\eqab}{\begin{eqnarray}}
\newcommand{\eqae}{\end{eqnarray}}

\newcommand{\bfmu}{{\boldsymbol\mu}}
\newcommand{\bfsigma}{{\boldsymbol\sigma}}
\newcommand{\bfSigma}{{\boldsymbol\Sigma}}
\newcommand{\bfOmega}{{\boldsymbol\Omega}}
\newcommand{\bflambda}{{\boldsymbol\lambda}}

\newcommand{\bfvareps}{{\boldsymbol\varepsilon}}
\newcommand{\bfeps}{{\boldsymbol\epsilon}}

\newcommand{\bfdelta}{{\boldsymbol\delta}}
\newcommand{\bfwp}{{\boldsymbol\wp}}

\newcommand{\bfA}{\mathbf{A}}
\newcommand{\bfW}{\mathbf{W}}
\newcommand{\bfV}{\mathbf{V}}
\newcommand{\bfM}{\mathbf{M}}
\newcommand{\bfD}{\mathbf{D}}
\newcommand{\bfC}{\mathbf{C}}
\newcommand{\bfX}{\mathbf{X}}
\newcommand{\bfY}{\mathbf{Y}}
\newcommand{\bfI}{\mathbf{I}}
\newcommand{\bfS}{\mathbf{S}}

\newcommand{\bfb}{\mathbf{b}}
\newcommand{\bfc}{\mathbf{c}}
\newcommand{\bfx}{\mathbf{x}}
\newcommand{\bfq}{\mathbf{q}}
\newcommand{\bff}{\mathbf{f}}
\newcommand{\bfdq}{\bfdelta\bfq}
\newcommand{\bfdb}{\bfdelta\bfb}
\newcommand{\bfdA}{\bfdelta\bfA}
\newcommand{\bfdAr}{\bfdelta\bfAr}

\newcommand{\bfAr}{{\bfA}_r}
\newcommand{\bfbr}{{\bfb}_r}
\newcommand{\bfcr}{{\bfc}_r}
\newcommand{\tbfAr}{\widetilde{\bfA}_r}

\newcommand{\tbfX}{\widetilde{\bfX}}
\newcommand{\tbfM}{\widetilde{\bfM}}
\newcommand{\tbfD}{\widetilde{\bfD}}
\newcommand{\tbfC}{\widetilde{\bfC}}

\newcommand{\tbfq}{\widetilde{\bfq}}

\smartqed  % flush right qed marks, e.g. at end of proof
\usepackage{graphicx}
\begin{document}

\title{Revisiting IRKA:\\ Connections with pole placement and backward stability\thanks{The work of Beattie was supported in parts by NSF through Grant DMS-1819110. 
{The work of Drma\v{c} was supported in parts by the the  DARPA  Contract  HR0011-16-C-0116 \emph{``On a Data-Driven, Operator-Theoretic Framework for Space-Time Analysis of Process Dynamics''} and the DARPA Contract HR0011-18-9-0033  \emph{``The Physics of Artificial Intelligence''}.}
The work of Gugercin was supported in parts by NSF through Grant DMS-1720257 and DMS-1819110.
}
}

\author{C.~Beattie        \and
        Z. ~Drma\v{c} \and
        S.~Gugercin 
        }


\institute{C.~Beattie \at
              Department of Mathematics, Virginia Tech, Blacksburg, VA, 24061, USA  \\
              \email{beattie@vt.edu}           %  \\
\and
Z. ~Drma\v{c} \at
              Faculty of Science, Department of Mathematics, University of Zagreb, Zagreb, 10000, Croatia  \\
              \email{drmac@math.hr}  
           \and
  S.~Gugercin \at
              Department of Mathematics and Division of Computational Modeling and Data Analytics, Virginia Tech, Blacksburg, VA, 24061, USA  \\
              \email{gugercin@vt.edu}    
}

\date{Received: date / Accepted: date}

\maketitle

\begin{abstract}
The iterative  rational  Krylov  algorithm  (\textsf{IRKA})  is  a popular  approach for producing locally optimal reduced-order $\mathcal{H}_2$-approximations to linear time-invariant (LTI) dynamical systems. Overall, \textsf{IRKA} has seen significant practical success in computing high fidelity (locally) optimal reduced models and has been successfully applied in a variety of large-scale settings. Moreover, \textsf{IRKA} has provided a foundation for recent extensions to the systematic model reduction of bilinear and nonlinear dynamical systems. 

Convergence of the basic \textsf{IRKA} iteration is generally observed to be rapid --- but not always; and despite the simplicity of the iteration, its convergence behavior is  remarkably complex and not well understood aside from a few special cases.  The overall effectiveness and computational robustness of the basic \textsf{IRKA} iteration is surprising since its algorithmic goals are very similar to a pole assignment problem, which can be notoriously ill-conditioned.   We investigate this connection here and discuss a variety of nice properties of the \textsf{IRKA} iteration that are revealed when the iteration is framed with respect to a primitive basis.   We find that the connection with pole assignment suggests refinements to the basic algorithm that can improve convergence behavior, leading also to new choices for termination criteria that assure backward stability.

\keywords{Interpolation \and Model Reduction \and $\mathcal{H}_2$-optimality \and Pole placement \and Backward stability}
\end{abstract}
\section{Introduction}
The iterative rational Krylov algorithm (\textsf{IRKA}) was introduced in \cite{Gugercin_Antoulas_Beattie:2008:IRKA} as an approach for producing locally optimal reduced-order $\mathcal{H}_2$-approximations to linear time-invariant (LTI) dynamical systems given, say, as
\begin{equation} \label{origLTIsys}
\dot{\mathbf{x}}(t) = \mathbf{A} \mathbf{x}(t) + \mathbf{b}\, u(t),~\quad
y(t) = \mathbf{c}^T \mathbf{x}(t),
\end{equation}
where $\mathbf{A}\in\mathbb{R}^{n\times n}$, and $\mathbf{b},\mathbf{c}\in
\mathbb{R}^n$. We will assume that 
the dynamical system is stable, i.e., all the eigenvalues of $\bfA$ have negative real parts. The cases of interest will be when $n$ is very large, and we seek a substantially lower order dynamical system, say, 
\begin{equation} \label{redLTIsys}
\dot{\mathbf{x}}_r(t) = \mathbf{A}_r \mathbf{x}_r(t) + \mathbf{b}_r\, u(t),\quad
y_r(t) = \mathbf{c}^T_r \mathbf{x}_r(t),
\end{equation}
with $\mathbf{A}_r\in\mathbb{R}^{r\times r}$, and $\mathbf{b}_r,\mathbf{c}_r\in
\mathbb{R}^r$.  One seeks a realization \eqref{redLTIsys} so that the reduced system order $r\ll n$ and the reduced system output $y_r\approx y$ uniformly well over all inputs $u\in \mathcal{L}_2$ with 
$\int_0^{\infty}|u(t)|^2\,dt\leq 1$.

Projection-based model reduction is a common framework to obtain reduced models: Given the full model \eqref{origLTIsys}, construct two model reduction bases 
$\bfV, \bfW \in \Cplx^{n \times r}$ with $\bfW^T\bfV$ invertible. Then the reduced model quantities in \eqref{redLTIsys} are given by  
\begin{equation} \label{Project}
    \bfAr =  (\mathbf{W}^T\mathbf{V})^{-1}\mathbf{W}^T\mathbf{A}\mathbf{V},~~\bfbr =(\mathbf{W}^T\mathbf{V})^{-1}\mathbf{W}^T\mathbf{b},~~ \mbox{and}~~\bfcr = \bfc \bfV.
\end{equation}
The following question arises: How to choose $\bfV$ and $\bfW$ so that the reduced model is a high-fidelity approximation to the original one? There are many different ways to construct $\bfV$ and $\bfW$, and we refer the reader to \cite{antoulas2005approximation,AntBG20,MORBook} for detailed descriptions of such methods for linear dynamical systems. Here we focus on constructing optimal interpolatory reduced models.

To discuss interpolation and optimality, we first need to define the concept of transfer function. Let $\mathcal{Y}(s)$, $\mathcal{Y}_r(s)$, and $\mathcal{U}(s)$ denote the Laplace transforms of 
$y(t)$, $y_r(t)$, and $u(t)$, respectively.
Taking the Laplace transforms of \eqref{origLTIsys} and \eqref{redLTIsys} yields
\begin{align}
\mathcal{Y}(s) &= H(s)\, \mathcal{U}(s)\quad \mbox{where}\quad H(s) = \bfc^T(s \bfI - \bfA)^{-1}\bfb,~\mbox{and} \\
\mathcal{Y}_r(s) &= H_r(s)\, \mathcal{U}(s)\quad \mbox{where}\quad H_r(s) = \bfc_r^T(s \bfI_r - \bfA_r)^{-1}\bfb_r. 
\end{align}
The rational functions $H(s)$ and $H_r(s)$ are    the transfer functions associated with the full model
\eqref{origLTIsys} and the reduced model \eqref{redLTIsys}. While $H(s)$ is a degree-$n$ rational function, $H_r(s)$ is of degree-$r$.  

Interpolatory model reduction aims to construct an $H_r(s)$ that  interpolates $H(s)$ at selected points in the complex plane. Indeed, we will focus on Hermite interpolation, as this will be tied to optimality later. Suppose we are given $r$ mutually distinct interpolation points (also  called \emph{shifts}), $\bfsigma = \{\sigma_1,\sigma_2,\ldots,\sigma_r\}$, in the complex plane. We will assume that  the shifts have positive real parts and that are closed (as a set) under conjugation, i.e., there exists an index permutation  $(i_1,\,i_2,\,\ldots,\,i_r\,)$ such that 
 $\overline{\bfsigma}=\{\overline{\sigma_1},\,\overline{\sigma_2},\, \ldots, \overline{\sigma_r}\}=\{\sigma_{i_1},\,\sigma_{i_2},\, \ldots, \sigma_{i_r}\}$.

Given $\bfsigma$, construct the model reduction bases $\bfV \in \Cplx^{n\times r}$ and $\bfW\in \Cplx^{n\times r}$ such that
\begin{align} \label{Veq}
\textsf{Range}(\bfV) &= \mathsf{span}\left\{(\sigma_1\bfI-\mathbf{A})^{-1}\mathbf{b},
 \ldots,(\sigma_r\bfI-\mathbf{A})^{-1}\mathbf{b}\right\}~~\mbox{and}\\\label{Weq}
 \textsf{Range}(\bfW) &= \mathsf{span}\left\{(\sigma_1\bfI-\mathbf{A}^T)^{-1}\mathbf{c},
 \ldots,(\sigma_r\bfI-\mathbf{A}^T)^{-1}\mathbf{c}\right\}.
\end{align}
Then, the reduced model \eqref{redLTIsys} constructed as in 
\eqref{Project} satisfies 
\begin{equation}\label{eq:Herm-interpol-sigma}
    H_r(\sigma_i) = H(\sigma_i) \quad \mbox{and} 
    \quad H'_r(\sigma_i) = H'(\sigma_i)\quad \mbox{for}~i=1,2,\ldots,r.
\end{equation}
In other words, $H_r(s)$ is a rational Hermite interpolant to $H(s)$ at the specified interpolation points. 
However, this construction requires knowing the interpolation points. How should one choose them to guarantee a high-fidelity reduced model?

We will measure fidelity using the $\mathcal{H}_2$ norm:
 The $\mathcal{H}_2$ norm of a dynamical system with transfer function $H(s)$ is defined as 
$$
\|H\|_{\mathcal{H}_2}= \sqrt{\frac{1}{2\pi}\int_{-\infty}^{\infty}|H(\Imunit\omega)|^2\,d\omega}.
$$
For the full model 
\eqref{origLTIsys} and the reduced model \eqref{redLTIsys}, 
the output error satisfies
$$
\| y  - y_r \|_{L_\infty} \leq \| H - H_r \|_{\mathcal{H}_2} 
\| u \|_{L_2},
$$
where $\| y - y_r\|_{L_\infty} = \sup_{t\geq 0} \mid y(t) - y_r(t)\mid$ and $\|u\|_{L_2} = \sqrt{\int_0^\infty \mid u(t) \mid^2 dt}$. So, a reduced model that minimizes 
the $\mathcal{H}_2$ distance $\| H - H_r \|_{\mathcal{H}_2}$ is guaranteed to yield uniformly good approximations over finite energy inputs.  Therefore, it is desirable  to find a reduced model with transfer function $H_r(s)$ that minimizes the $\mathcal{H}_2$ distance, i.e., to find $H_r(s)$
such that
$$
\|H-H_r\|_{\mathcal{H}_2}= \min_{\substack{G_r\,\mbox{\tiny{stable}}\\ \mbox{\tiny{order} }G_r\leq r}}\|H-G_r\|_{\mathcal{H}_2}
$$
at least locally in a neighborhood of $H_r$. This is a heavily studied topic; see, e.g.,
\cite{wilson1970optimum,hyland1985theoptimal,zigic1993contragredient,spanos1992anewalgorithm,benner_sylvester}
for Sylvester-equation formulation 
 and
  \cite{meieriii1967approximation, Gugercin_Antoulas_Beattie:2008:IRKA, vandooren2008hom,bunse-gerstner2009hom,beattie2012realization,krajewski1995program,xu2010optimal}
 for interpolation formulation. Indeed, these two formulations are equivalent as shown in \cite{Gugercin_Antoulas_Beattie:2008:IRKA} and we focus on the interpolatory formulation.

How does the $\mathcal{H}_2$ optimality relate to Hermite interpolation? Let $\mu_1,\ldots,\mu_r$ be the eigenvalues of $\bfAr$, assumed simple. If $H_r(s)$ is an $\mathcal{H}_2$-optimal approximation to $H(s)$, then it is a Hermite interpolant to $H(s)$ at the points 
$\sigma_i = -\mu_i$, i.e.,
\begin{equation}\label{eq:Hermite-interpol}
    H_r(-\mu_i) = H(-\mu_i) \quad\mbox{and}\quad  H'_r(-\mu_i) = H'(-\mu_i),\quad \mbox{for}~i=1,2,\ldots,r.
\end{equation}
These conditions are known as Meier-Luenberger conditions for optimality \cite{meieriii1967approximation}. However, one cannot simply use $\sigma_i = -\mu_i$ in constructing $\bfV$ and $\bfW$ in \eqref{Veq}-\eqref{Weq} since $\mu_i$s are not known a priori. This requires an iteratively corrected algorithm. The iterative rational Krylov algorithm \textsf{IRKA} \cite{Gugercin_Antoulas_Beattie:2008:IRKA} as outlined in  Algorithm  \ref{ALG-HOKIE} precisely achieves this task. It reflects the intermediate interpolation points until the required optimality criterion, i.e., $\sigma_i = -\mu_i$ is met.  Upon convergence, the reduced model  is a locally optimal $\mathcal{H}_2$-approximation to \eqref{redLTIsys}. \textsf{IRKA} has been  successful in producing locally optimal reduced models at a modest cost and many variants have been proposed; see, e.g., \cite{beattie2007kbm,beattie2009trm,beattie2012realization,panzer2013,HokMag2018,poussot2011iterative,vuillemin2013h2,breiten2013near,gugercin2008isk,goyal2019time}. Moreover, it has been successfully extended to model reduction of bilinear \cite{breiten12,flagg15} and quadratic-bilinear systems \cite{benner2018h2}, two important classes of structured nonlinear systems.

\begin{algorithm}[hh] 
\caption{$(\mathbf{A}_r, \mathbf{b}_r, \mathbf{c}_r)=
\mathsf{IRKA}(\mathbf{A},\mathbf{b},\mathbf{c},r)$} \label{ALG-HOKIE}
\begin{algorithmic}[1]
\STATE Initialize shifts ${\bfsigma}=(\sigma_1,\sigma_2,\ldots,\sigma_r) \subset \mathbb{C}_{+}\equiv\{z\in\mathbb{C}\; :\; \mathrm{Re}(z) > 0\}$ that are closed (as a set) under under conjugation;
\REPEAT
\STATE
 Compute a basis of  $\mathsf{span}((\sigma_1\bfI-\mathbf{A})^{-1}\mathbf{b},
 \ldots,(\sigma_r\bfI-\mathbf{A})^{-1}\mathbf{b})$ $\rightarrow$ $\mathbf{V}$;
\STATE
 Compute a basis of  $\mathsf{span}((\sigma_1\bfI-\mathbf{A}^T)^{-1}\mathbf{c},
 \ldots,(\sigma_r\bfI-\mathbf{A}^T)^{-1}\mathbf{c})$ $\rightarrow$ $\mathbf{W}$;
\STATE $\mathbf{A}_r = (\mathbf{W}^T\mathbf{V})^{-1}\mathbf{W}^T\mathbf{A}\mathbf{V}$;
\STATE Compute the eigenvalues (reduced poles) $\bflambda(\mathbf{A}_r)=(\lambda_1(\mathbf{A}_r),\ldots,\lambda_r(\mathbf{A}_r))$ ;
\STATE Compute the (matching) distance $\zeta$ between the sets $\bflambda(\mathbf{A}_r)$ and $-\bfsigma$ ;
\STATE $\sigma_i \longleftarrow -\lambda_i(\mathbf{A}_r)$, $i=1,\ldots, r$ ;  
\UNTIL{$\zeta$ \emph{sufficiently small}}
\STATE $\mathbf{b}_r=(\mathbf{W}^T\mathbf{V})^{-1}\mathbf{W}^T\mathbf{b}$ ; $\mathbf{c}_r=\mathbf{V}^T\mathbf{c}$ ;
\STATE The reduced order model is $(\mathbf{A}_r,\mathbf{b}_r,\mathbf{c}_r)$.
\end{algorithmic}
\end{algorithm}
Our goal in this paper is not to compare model reduction techniques,  nor is it to illustrate effectiveness of \textsf{IRKA} and its variants. Refer to sources cited above for such analyses. Our main goal here is to revisit \textsf{IRKA} in its original form and reveal new connections to the pole placement problem (Section \ref{sec:poleplacement}) by a thorough analysis of the quantities involved in a special basis (Section \ref{sec:prim_basis}). This will lead to a backward stability formulation relating then to new stopping criteria (Section \ref{sec:backstab}). 

 In order to keep the discussion concise, we focus here on single-input/single-output dynamical systems, i.e., $u(t),y(t),y_r(t) \in \Real$. For detailed discussion of 
$\mathcal{H}_2$-optimal model reduction in the complementary  multi-input/multi-output case, see \cite{AntBG20,MORBook,abg10}.

\section{Structure in the primitive bases} \label{sec:prim_basis}
In Steps 3 and 4 of  \textsf{IRKA} as laid out in Algorithm \ref{ALG-HOKIE}~above, the matrices $\bfV$ and $\bfW$ are each chosen as bases for a pair of rational Krylov subspaces. The reduced model is independent of the particular bases chosen and one usually constructs them to be orthonormal.  We consider a different choice in this section, and show that if $\bfV$
and $\bfW$ are chosen instead as
\emph{primitive} bases, i.e., if  
\begin{equation} \label{primbasis}
\bfV=\begin{bmatrix} (\sigma_1 \mathbf{I} - \bfA)^{-1}\bfb & \ldots & (\sigma_r \mathbf{I} - \bfA)^{-1}\bfb\end{bmatrix} \quad \mbox{and} \quad
\bfW=\begin{bmatrix} (\sigma_1 \mathbf{I} - \bfA)^{-T}\bfc & \ldots & (\sigma_r \mathbf{I} - \bfA)^{-T}\bfc\end{bmatrix},
\end{equation}
then the state-space realization of the reduced model exhibits an important structure which forms the foundation of our analysis that follows in
Sections \ref{sec:poleplacement} and \ref{sec:backstab}. Therefore, in the rest of the paper, we use primitive bases for $\bfV$ and $\bfW$ as
given in \eqref{primbasis}. We emphasize that this does not change the reduced model $H_r(s)$; it is simply a change of basis that reveals nontrivial structure that can be exploited both in the theoretical analysis of the algorithm and for its
efficient software implementation. 
 
It is easy to check that (\!\cite{Gugercin_Antoulas_Beattie:2008:IRKA}), for $\bfV$ and $\bfW$ as primitive bases 
\eqref{primbasis}, 
the matrices $\bfW^{T}\bfA \bfV$ and
$\bfW^{T}\bfV$ are symmetric; but not necessarily Hermitian. Moreover, one may 
directly verify that (\!\cite{abg10}),
$\bfW^T\bfV$ is the Loewner matrix whose $(i,j)$th entry, for $i,j=1,\ldots,r$, is given by
\begin{equation}\label{eq:W^T*V-Loewner}
    (\bfW^{T}\bfV)_{ij}=
    [\sigma_{i},\sigma_{j}]H
    := \frac{H(\sigma_i)-H(\sigma_j)}{\sigma_i-\sigma_j},
    \end{equation}
    with the convention that $[\sigma_{i},\sigma_{i}]H = H'(\sigma_i)$.

\begin{lemma} \label{RatKryCanonForm}
Let
$\omega_{r}(z)=(z-\sigma_{1})(z-\sigma_{2})\ldots(z-\sigma_{r})$ be
the nodal polynomial associated with the shifts $\bfsigma = \{\sigma_1,\sigma_2,\ldots,\sigma_r\}$.
For any monic polynomial $p_{r}\in \mathcal{P}_{r}$, define the vector
$$
\bfq=(q_1,\ldots, q_r)^T,\;\; q_{i}=\frac{p_{r}(\sigma_{i})}{\omega_{r}'(\sigma_{i})},\;\;
i=1,\,\ldots,\,r,
$$
and the matrix $\bfAr=\bfSigma_{r}-{\bfq}\mathbf{e}^{T}$
with $\bfSigma_{r}=diag(\sigma_{1},\ldots,\sigma_{r})$. Then $\det(z\bfI-{\bfA}_{r})=p_{r}(z)$
and
\begin{eqnarray} 
    {\bfA}\bfV-\bfV{\bfA}_{r}
      &= & -p_{r}({\bfA})[\omega_{r}({\bfA})]^{-1}{\bfb}\mathbf{e}^{T} ,
    \label{ratKryRightV} \\[2mm]
    \bfW^{T}{\bfA}-{\bfA}_{r}^{T}\bfW^{T}
      &=& -\mathbf{e}{\bfc}^{T}p_{r}({\bfA})[\omega_{r}({\bfA})]^{-1}.
    \label{ratKryLeftW}
\end{eqnarray}
\end{lemma}
\begin{proof}
Pick any index $1\leq k \leq r$ and consider $f_{k}(z)=p_{r}(z)-z\cdot
\prod_{i\neq k}(z-\sigma_{i})$.  Evidently, $f_{k}\in \mathcal{P}_{r-1}$
and so the Lagrange interpolant on $\sigma_{1},\,\sigma_{2},\,\ldots,\,\sigma_{r}$
is exact:
$$
f_{k}(z) =
p_{r}(z)-z\cdot
\prod_{i\neq k}(z-\sigma_{i}) =\sum_{i=1}^{r}f_{k}(\sigma_{i})\frac{\omega_{r}(z)}{(z-\sigma_{i})\omega_{r}'(\sigma_{i})} .
$$
Divide by $\omega_{r}(z)$ and rearrange to obtain
\begin{equation}
    \frac{z}{\sigma_{k}-z} -
    \sum_{i=1}^{r}
    \,\left(-\frac{f_{k}(\sigma_{i})}{\omega_{r}'(\sigma_{i})}\right)
    \frac{1}{\sigma_{i}-z}= -\frac{p_{r}(z)}{\omega_{r}(z)} .
    \label{LagrangeInterp}
\end{equation}
Let $\Gamma$ be a Jordan curve that separates $\mathbb{C}$ into two
open, simply-connected sets, $\mathcal{C}_{1},\, \mathcal{C}_{2}$
with $\mathcal{C}_{1}$ containing all the eigenvalues of
${\bfA}$ and $\mathcal{C}_{2}$ containing both the point at
$\infty$ and the shifts $\{\sigma_{1},\ldots,\sigma_{r}\}$.
For any function $f(z)$ that is analytic in a compact set containing
$\mathcal{C}_{1}$, $f({A})$ can be
defined as
$
f({\bfA})=\frac{1}{2\pi\Imunit}\int_{\Gamma}f(z)\,(z\bfI-{\bfA})^{-1}\, dz .
$
Applying this to (\ref{LagrangeInterp}) gives
$$
{\bfA}(\sigma_{k}\bfI-{\bfA})^{-1}-
\sum_{i=1}^{r}
\,\left(-\frac{f_{k}(\sigma_{i})}{\omega_{r}'(\sigma_{i})}\right)
(\sigma_{i}\bfI-{\bfA})^{-1} =-p_{r}({\bfA})[\omega_{r}({\bfA})]^{-1} .
$$
Postmultiplication by ${\bfb}$ provides the $k^{th}$ column of
(\ref{ratKryRightV}), while premultiplication by ${\bfc}^{T}$
(and since $(\sigma_{k}\bfI-{\bfA})^{-1}$
commutes with ${\bfA}$) provides
the $k^{th}$ row of (\ref{ratKryLeftW}).

To compute the characteristic polynomial of $\bfA_r$, we use the alternative
factorizations\footnote{For the reader's convenience, here we actually reproduce the proof of the Sherman--Morrison determinant
formula.},
{
\begin{align*}
 \begin{pmatrix}
    z\bfI-\bfSigma_{r} & {\bfq}   \\
    \mathbf{e}^{T}  & -1
\end{pmatrix}= &\begin{pmatrix}
    \bfI & -{\bfq}  \\
    \mbox{\boldmath{$0$}}^{T} & 1
\end{pmatrix}\,
\begin{pmatrix}
    z\bfI-{\bfA}_{r} & \mbox{\boldmath{$0$}}  \\
    \mbox{\boldmath{$0$}}^{T} & -1
\end{pmatrix}\, \begin{pmatrix}
    \bfI & \mbox{\boldmath{$0$}}  \\
    -\mathbf{e}^{T} & 1
\end{pmatrix}  \\
\begin{pmatrix}
    z\bfI-\bfSigma_{r} & {\bfq}   \\
    \mathbf{e}^{T}  & -1
\end{pmatrix}= & \begin{pmatrix}
    \bfI & \mbox{\boldmath{$0$}}  \\
    \mathbf{e}^{T}(z\bfI-\bfSigma_{r})^{-1} & 1
\end{pmatrix}\,
\begin{pmatrix}
    z\bfI-\bfSigma_{r} & \mbox{\boldmath{$0$}}  \\
    \mbox{\boldmath{$0$}}^{T} &
    -a(z)
\end{pmatrix}\,\begin{pmatrix}
    \bfI & (z\bfI-\bfSigma_{r})^{-1} {\bfq} \\
    \mbox{\boldmath{$0$}} & 1
\end{pmatrix},
\end{align*}
}
where $a(z)=1+\mathbf{e}^{T}(z\bfI-\Sigma_{r})^{-1}{\bfq}$.
Then we have that
\begin{eqnarray}
   \det(z\bfI-{\bfA}_{r}) &= & \det(z\bfI-\bfSigma_{r})
    \cdot a(z)
  =\omega_{r}(z)\cdot\left(1+\mathbf{e}^{T}(z\bfI-
  \bfSigma_{r})^{-1}
  {\bfq}\right) \label{secularEqn} \\
  &=& \omega_{r}(z)
  +\sum_{i=1}^{r} p_{r}(\sigma_{i})\left(\frac{\omega_{r}(z)}
  {\omega_{r}'(\sigma_{i})\,(z-\sigma_{i})}\right)= p_{r}(z) , \nonumber
\end{eqnarray}
where the last equality follows by observing that the penultimate
expression describes a monic polynomial of degree $r$ that interpolates $p_{r}$
at $\sigma_{1},\, \sigma_{2},\, \ldots\,\sigma_{r}$. \hfill$\Box$
\end{proof}

Lemma \ref{RatKryCanonForm}, and more specifically 
\eqref{ratKryRightV} and \eqref{ratKryLeftW}, reveal the rational Krylov structure arising from the choice of $\bfV$ and $\bfW$ in \eqref{primbasis}. At first, connection of the involved quantities such as the vector $\bfq$ to \textsf{IRKA} quantities might not be clear. In the following result, we make these connections precise.
\begin{lemma}
    \label{RatKryCanFrmCor}
    Let the reduced model $H_r(s) = \bfcr^T(s\bfI-\bfA_r)^{-1}\bfbr$ be obtained by projection as in \eqref{Project} using
    the primitive bases $\bfW$ and $\bfV$ in \eqref{primbasis}. 
   Let
    $$p_{r}(z)=\det(z\, \bfW^{T}\bfV-\bfW^{T}\bfA \bfV )/\det(\bfW^{T}\bfV).$$
    Then \eqref{ratKryRightV} and \eqref{ratKryLeftW} hold with
    $\bfA_{r}=(\bfW^{T}\bfV)^{-1}\,\bfW^{T}\bfA \bfV$. In particular, $\bfb_r = (\bfW^{T}\bfV)^{-1}\,\bfW^{T}\bfb = \bfq$ and
    $\bfA_r=(\bfW^{T}\bfV)^{-1}\,\bfW^{T}\bfA\bfV =
    \bfSigma_r - \bfe \bfq^T$.
    Moreover, if $\mu_{\ell}$ is an eigenvalue of
    $\bfAr$, then
    \begin{equation}
        \mathbf{x}_{\ell}=\left(\bfSigma_{r}-\mu_{\ell}\bfI\right)^{-1}\mathbf{q}
        \label{rightRitzvector}
    \end{equation}
    is an associated (right) eigenvector of $\bfA_{r}$ for $\ell = 1,2,\ldots,r$. Similarly, the vector
    \begin{equation}
        \mathbf{x}_{\ell}^{T} \bfW^{T} \bfV=\nu_{\ell}\mathbf{e}^{T}\left(\bfSigma_{r}-\mu_{\ell}\bfI\right)^{-1}
        \label{leftRitzvector}
    \end{equation}
    is an associated left eigenvector of $\bfA_{r}$, for $\ell = 1,2,\ldots,r$, with
    $\displaystyle \nu_{\ell}=\hat{\phi}_{\ell}\frac{p_{r}'(\mu_{\ell})}{\omega_{r}(\mu_{\ell})}$, and
    $\hat{\phi}_{\ell}$ is the residue of $H_{r}(s)$ at $s=\mu_{\ell}$, i.e.,
    $\hat{\phi}_{\ell} = \lim_{s\to \mu_\ell}(s-\mu_\ell)H_r(s)$.
\end{lemma}

\begin{proof} 
Now, choose a monic polynomial
$\hat{p}_{r}\in \mathcal{P}_{r}$ so that $\bfW^{T}\hat{p}_{r}({\bfA})
[\omega_{r}({\bfA})]^{-1}\mathbf{b}=0$.  Then (\ref{ratKryRightV})
and (\ref{ratKryLeftW}) hold with an associated $\bfA_{r}
=\bfSigma_{r}-\mathbf{q}\mathbf{e}^{T}$ as given in
Lemma \ref{RatKryCanonForm}.  But then applying ${\bfW}^{T}$ to \eqref{ratKryRightV}
yields 
$\bfA_{r}=(\bfW^{T} \bfV)^{-1}\,\bfW^{T}\bfA \bfV$. This in turn
implies $\hat{p}_{r}(z)=p_{r}(z)$. 

Suppose that $\mu_{\ell}$ is an eigenvalue of $\bfA_{r}$.
Directly substitute $\mathbf{x}_{\ell}=
\left(\bfSigma_{r}-\mu_{\ell}\bfI\right)^{-1}{\bfq}$ and use
\eqref{secularEqn} with $z=\mu_{\ell}$ to obtain
\begin{eqnarray*}
\bfA_{r}\mathbf{x}_{\ell}&=& \left(\bfSigma_{r}-{\bfq}\mathbf{e}^{T}\right)
\left(\bfSigma_{r}-\mu_{\ell}\bfI\right)^{-1}{\bfq}
 =\left(\Sigma_{r}-\mu_{\ell}\bfI-{\bfq}\mathbf{e}^{T}
 +\mu_{\ell}\bfI\right)
\left(\bfSigma_{r}-\mu_{\ell}\bfI\right)^{-1}{\bfq} \\
   & = & {\bfq}\left(1-\mathbf{e}^{T}
   \left(\bfSigma_{r}-\mu_{\ell}\bfI\right)^{-1}{\bfq}\right)
   +\mu_{\ell}\left(\bfSigma_{r}-\mu_{\ell}\bfI\right)^{-1}{\bfq}
   =\mu_{\ell}\mathbf{x}_{\ell}.
\end{eqnarray*}
Thus, $\mathbf{x}_{\ell}$ is a right eigenvector for $\bfA_{r}$
associated with $\mu_{\ell}$. Note that
$\mathbf{x}_{\ell}$ also solves the generalized
eigenvalue problems:
\begin{equation}
   \ \mbox{(a)}\  \bfW^{T}\bfA \bfV\mathbf{x}_{\ell}
   =\mu_{\ell}\bfW^{T}\bfV\mathbf{x}_{\ell}
    \quad \mbox{and}\quad\mbox{(b)}\
    \mathbf{x}_{\ell}^{T}\bfW^{T}\bfA \bfV
    =\mu_{\ell}\ \mathbf{x}_{\ell}^{T}\bfW^{T}\bfV.\
    \label{eigval}
\end{equation}
(\ref{eigval}a) is immediate from the definition of $\bfA_{r}$.
(\ref{eigval}b) is obtained by transposition of (\ref{eigval}a) and using the facts that
$\bfW^{T}\bfA \bfV$ and
$\bfW^{T}\bfV$ are symmetric.
Notice that (\ref{eigval}b) shows that
$\mathbf{x}_{\ell}^{T}\bfW^{T}\bfV$ is a left eigenvector for $\bfA_{r}$
associated with $\mu_{\ell}$.  On the other hand, direct substitution also
shows that
$$
\left[\mathbf{e}^{T}\left(\bfSigma_{r}-\mu_{\ell}\bfI\right)^{-1}\right]
\bfA_{r}
=\left[\mathbf{e}^{T}\left(\bfSigma_{r}-\mu_{\ell}\bfI\right)^{-1}\right]
\left(\bfSigma_{r}-\mathbf{q}\mathbf{e}^{T}\right)=
\mu_{\ell}\left[\mathbf{e}^{T}\left(\bfSigma_{r}-\mu_{\ell}
\bfI\right)^{-1}\right],
$$
so $\mathbf{e}^{T}\left(\bfSigma_{r}-\mu_{\ell}\bfI\right)^{-1}$ is
also a left eigenvector of $\bfA_{r}$ associated with $\mu_{\ell}$.
We must have then
$$
\mathbf{x}_{\ell}^{T}\bfW^{T}\bfV=\nu_{\ell}\ \mathbf{e}^{T}
\left(\bfSigma_{r}-\mu_{\ell}\bfI\right)^{-1}
$$
for some scalar $\nu_{\ell}$ which we now determine. 
Using (\ref{eq:W^T*V-Loewner}) and (\ref{rightRitzvector}), the $j^{th}$ component of each side of
the equation can be
expressed as
\begin{equation}
\left(\mathbf{x}_{\ell}^{T}\bfW^{T}\bfV\right)_{j}=
\sum_{i=1}^{r}\frac{p_{r}(\sigma_{i})}{\omega_{r}'(\sigma_{i})}
\frac{[\sigma_{i},\sigma_{j}]H_{r}}{\sigma_{i}-\mu_{\ell}}
=\frac{\nu_{\ell}}{\sigma_{j}-\mu_{\ell}}.
    \label{nuDef}
\end{equation}
Define the function $f(z)=p_{r}(z)\, [z,\sigma_{j}]H_{r}$.  It is easily checked that $f(z)$ is
a polynomial of degree $r-1$ and so Lagrange interpolation on
$\{\sigma_{1},\,\sigma_{2},\,\ldots,\,\sigma_{r} \}$ is exact:
$$
f(z)=\sum_{i=1}^{r} \left(p_{r}(\sigma_{i})\,
[\sigma_{i},\sigma_{j}]H_{r}\right)\,
\frac{\omega_{r}(z)}{(z-\sigma_{i})\,\omega_{r}'(\sigma_{i})} .
$$
Now evaluate this expression at $z=\mu_{\ell}$:
$$
\sum_{i=1}^{r}
\frac{\left(p_{r}(\sigma_{i})\, [\sigma_{i},\sigma_{j}]H_{r}\right)\,
\omega_{r}(\mu_{\ell})}{(\mu_{\ell}-\sigma_{i})\,\omega_{r}'(\sigma_{i})}
= f(\mu_{\ell})=\lim_{z\rightarrow \mu_{\ell}}p_{r}(z)\,
    \frac{H_{r}(z)-H_{r}(\sigma_{j})}{z-\sigma_{j}}
    = \frac{p_{r}'(\mu_{\ell})\, \hat{\phi}_{\ell}}{\mu_{\ell}-\sigma_{j}},
$$
where we observe that
$\lim_{z\rightarrow\mu_{\ell}}p_{r}(z)H_{r}(z)
=\hat{\phi}_{\ell}\prod_{i\neq\ell}(\mu_{\ell}-\mu_{i})\
=\hat{\phi}_{\ell}\,p_{r}'(\mu_{\ell})$.  Comparing this expression
to (\ref{nuDef}) we find $\displaystyle \nu_{\ell}=\hat{\phi}_{\ell}
\frac{p_{r}'(\mu_{\ell})}{\omega_{r}(\mu_{\ell})}$. \hfill$\Box$
\end{proof}

Lemma \ref{RatKryCanFrmCor} illustrates that 
if the primitive bases $\bfV$ and $\bfW$ in \eqref{primbasis}  are employed in \textsf{IRKA}, then the reduced matrix $\bfAr$ at every iteration step
is  a rank-$1$ perturbation of the diagonal matrix of shifts. This matrix $\bfA_r = \bfSigma_r - \bfq\bfe^T$ is known as the generalized companion matrix. This special structure allows explicit computation of the left and right eigenvectors of
$\bfAr$ as well. The next corollary gives further details about the spectral decomposition of $\bfAr$.

\begin{corollary}  \label{EVDArArtCor}
Consider the setup in Lemma  \ref{RatKryCanFrmCor}.
Define the $r \times r$ Cauchy matrix $\bfC = \bfC(\bfsigma,\bfmu)$ as
\begin{equation} \label{CauchyC}
    \bfC_{ij} = \frac{1}{\sigma_i - \mu_j}\quad \mbox{for}~
    i,j=1,2,\ldots,r,
\end{equation}
and the $r \times r$ diagonal matrix $\bfD_\bfq = \mathrm{diag}(q_1,q_2,\ldots,q_r).$
Then $\bfAr = \bfSigma - \bfq\bfe^T$ has the
 spectral
decomposition  
\begin{equation} \label{EVDAr}
\bfAr = \bfX \bfM 
{\bfX}^{-1}\quad
\mbox{where}~
\bfM=\mathrm{diag}(\mu_1,\mu_2,\ldots,\mu_r)~\mbox{and}~
{\bfX}=\bfD_{\bfq}{\bfC}.
\end{equation}
Moreover, $\bfAr^T = \bfD_{{\bfq}}^{-1} {\bfAr}
\bfD_\bfq$ and its spectral decomposition is 
\begin{equation} \label{EVDArt}
\bfAr^T =  \bfSigma - \bfe\bfq^T = 
\bfD_{{\bfq}}^{-1} \bfAr \bfD_{{\bfq}} = 
\bfC \bfM \bfC^{-1}.
\end{equation}
\end{corollary}
\begin{proof}
The spectral decomposition of $\bfAr$ in \eqref{EVDAr} directly follows from  \eqref{rightRitzvector} by observing that
$
\bfx_\ell = (\bfSigma_r - \mu_\ell \bfI)^{-1}\bfq = 
\begin{bmatrix}
\frac{q_1}{\sigma_1-\mu_\ell} &
\frac{q_2}{\sigma_2-\mu_\ell} &
\cdots &
\frac{q_r}{\sigma_r-\mu_\ell}
\end{bmatrix}^T.$ Therefore, 
the eigenvector matrix $\bfX = \begin{bmatrix}
\bfx_1 & \bfx_2 & \ldots & \bfx_r\end{bmatrix}$ can be written as 
$\bfX = \bfD_\bfq \bfC$, proving \eqref{EVDAr}. The spectral decomposition of $\bfAr^T$ in \eqref{EVDArt} can be proved similarly using 
\eqref{leftRitzvector}, i.e., the fact that $(\bfSigma_r - \mu_\ell \bfI)^{-T}\bfe$ is an eigenvector of $\bfAr^T$. Finally,
$\bfD_\bfq^{-1}\bfA_r \bfD_\bfq = \bfD_q^{-1}(\bfSigma_r - \bfq \bfe^T) \bfD_\bfq = \bfSigma_r - \bfD_q^{-1} \bfq \bfe^T \bfD_\bfq$ since
both $\bfD_q$ and $\bfSigma_r$ are diagonal. Moreover, it follows from the definition of $\bfD_\bfq = \mathrm{diag}(q_1,q_2,\ldots,q_r)$ that
$\bfD_\bfq^{-1} \bfq = \bfe$ and $\bfe^T \bfD_\bfq = \bfq$, thus completing the proof. \hfill$\Box$
\end{proof}
%%%%%%%%%%%%%%%%%%%%%%%%%%%%%%%%%%%%%%%%%%%%%%%%%%%%%%%%%%%

\section{A pole placement connection}\label{sec:poleplacement}
The main goal of this paper is to reveal the structure of the iterations in Algorithm \ref{ALG-HOKIE}, and in particular to study the limiting behaviour of the sequence of the shifts $\bfsigma^{(k)}$, $k=1, 2, \ldots$. In this section, we explore an intriguing idea to recast the computation of the shifts in Algorithm \ref{ALG-HOKIE} in a pole placement framework, and then to examine its potential for improving the convergence.

As Lemma \ref{RatKryCanFrmCor} illustrates, if the primitive bases  \eqref{primbasis}  are employed in \textsf{IRKA}, then at every step of \textsf{IRKA}, we have  $\bfA_r = \bfSigma_r - \bfq\bfe^T$. Then, in the $k$--th step, we start with the shifts $\sigma_1^{(k)},\ldots, \sigma_r^{(k)}$ and use them to
build the matrix
\begin{equation}\label{eq::A_r^k+1}
\bfA_r^{(k+1)}(\bfsigma^{(k)}) = \mathrm{diag}\left(\sigma_1^{(k)},\ldots,\sigma_r^{(k)}\right) - \bfq^{(k+1)}\bfe^T ,
\end{equation}
where the vector $\bfq^{(k+1)}$ (the reduced input in step $k$) ensures that the Hermite interpolation conditions are fulfilled, see (\ref{eq:Herm-interpol-sigma}).
If $\sigma_i^{(k)}$ is real, then $q_i^{(k+1)}$ is real as well; if for some $i\neq j$ $\sigma_i^{(k)} = \overline{\sigma_j^{(k)}}$, then $q_i^{(k+1)} = \overline{q_j^{(k+1)}}$. As a consequence, $\bfA_r^{(k+1)}$ is similar to a real matrix and its eigenvalues will remain closed under complex conjugation. Further, if some $q_i^{(k+1)}=0$, then the corresponding $\sigma_i^{(k)}$ is an eigenvalue of $\bfA_r^{(k+1)}(\bfsigma^{(k)})$; thus, if we assume that the shifts are in the open right half-plane and that  $\bfA_r^{(k+1)}(\bfsigma^{(k)})$ is stable, then $q_i^{(k+1)}\neq 0$ for all $i$.
Then, the new set of shifts is defined as 
\begin{equation}\label{eq:IRKA-map}
 \bfsigma^{(k+1)} = - \bfmu^{(k+1)},\;\;\mbox{where}\;\; \bfmu^{(k+1)} = \mathrm{eig}(\bfA_r^{(k+1)}(\bfsigma^{(k)})), 
\end{equation}
where $\mathrm{eig}(\cdot)$ is a numerical algorithm that computes the eigenvalues and returns them in some order.\footnote{Since the matrix is a rank one perturbation of the diagonal matrix, all eigenvalues can be computed in $O(r^2)$ operations by specially tailored algorithms.} In the limit as $k\rightarrow\infty$ the shifts should satisfy 
(\ref{eq:Herm-interpol-sigma}) and (\ref{eq:Hermite-interpol}).

%%%%%%%%%%%%%%%%%%%%%%%%%%%%%%%%%%%%%%%%%%%%%%%%%%%%%%%%%%%%%%%

\subsection{Measuring numerical convergence}\label{SS=Matching+example}
Numerical convergence in an implementation of Algorithm \ref{ALG-HOKIE} is declared if $\bfsigma^{(k+1)} \approx\bfsigma^{(k)}$, where the distance between two consecutive sets of shifts is measured using the optimal matching\footnote{The shifts (eigenvalues) are naturally considered as equivalence classes in $\mathbb{C}^r/\mathbb{S}_r$.}
$$
d(\bfsigma^{(k+1)}, \bfsigma^{(k)} ) =
\min_{\pi\in\mathbb{S}_r}\max_{i=1:r}|\bfsigma^{(k+1)}_{\pi(i)} -
\bfsigma^{(k)}_i| ,\;\;\mbox{where $\mathbb{S}_r$ denotes the symmetric group.}
$$
In an implementation, it is convenient to use the easier to compute Hausdorff distance 
$$
h(\bfsigma^{(k+1)}, \bfsigma^{(k)} ) = \max\left\{ 
\max_{j}\min_{i} \left|\bfsigma^{(k+1)}_{j} -
\bfsigma^{(k)}_i\right| ,  \max_{i}\min_{j} \left|\bfsigma^{(k)}_{i} -
\bfsigma^{(k+1)}_j\right|\right\},
$$
for which $h(\bfsigma^{(k+1)}, \bfsigma^{(k)} ) \leq d(\bfsigma^{(k+1)}, \bfsigma^{(k)} )$, so the stopping criterion (Line 9. in Algorithm \ref{ALG-HOKIE}) must be first satisfied in the Hausdorff metric.

Numerical evidence shows that many scenarios are possible during the iterations in Algorithm \ref{ALG-HOKIE} -- from swift to slow convergence. Characterizing the limit behavior in general is an open problem; in the case of symmetric system local convergence is established in \cite{FLAGG2012688}. Moreover, we have also encountered miss--convergence in form of the existence of at least two accumulation points that seem to indicate existence of periodic points of the mapping (\ref{eq:IRKA-map}). This is illustrated in the following example.
\begin{example}\label{EX-CDP-loop}
We take the matrix $\bfA\in \Real^{120 \times 120}$ from the CD player benchmark example \cite{ChaV2005,DraSB92}  from the NICONET  benchmark collection \cite{NICONET}
and set $\bfb=\bfc=\bfe$. With a particular set of $r=29$ initial shifts, we obtained separate behaviours for the odd and the even iterates, as shown in Figure \ref{FIG_looped-shifts}.  
\begin{figure}[!ht]
\begin{center}
\includegraphics[width=0.79\linewidth, height = 3.2in]{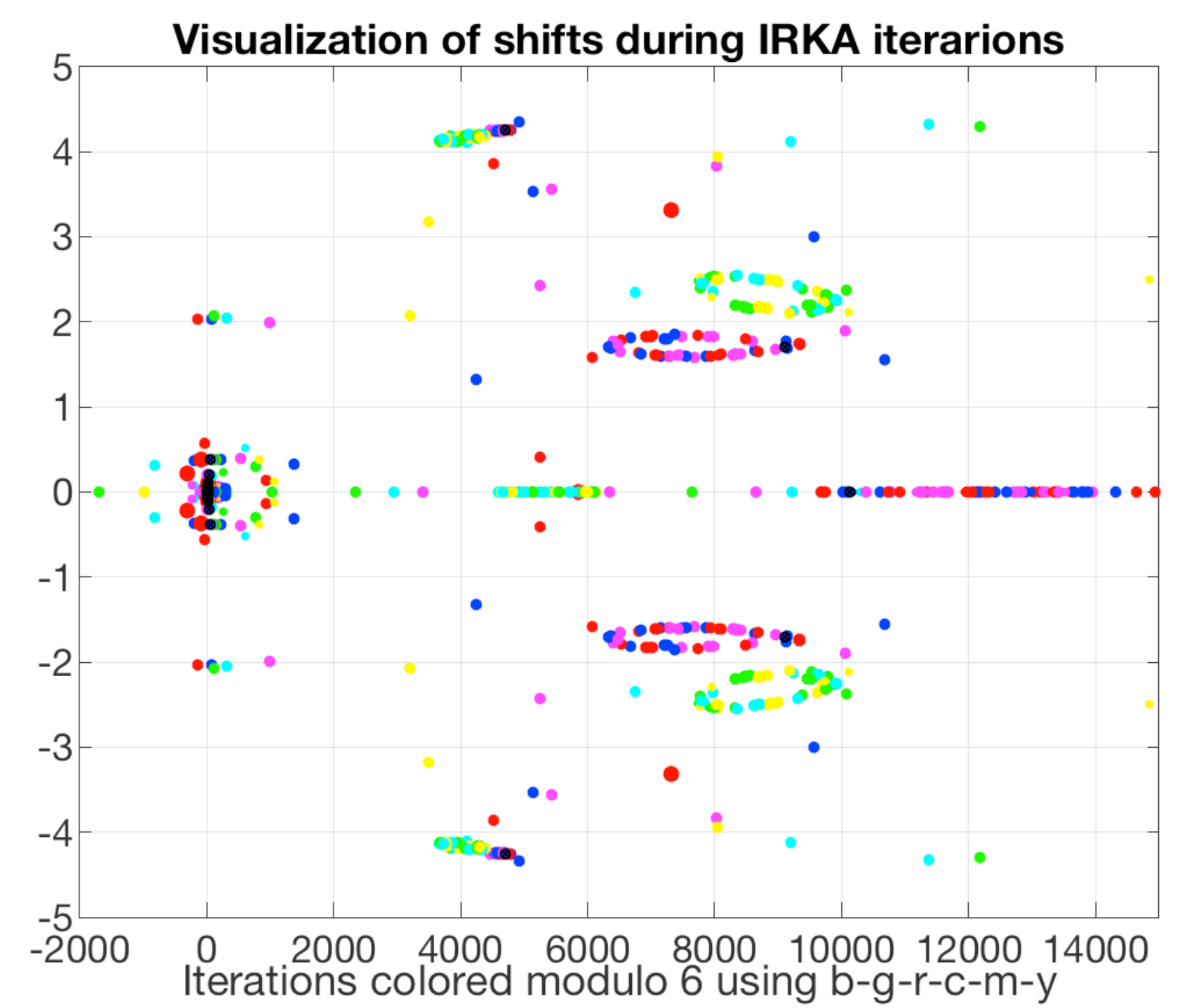}
\end{center}
\caption{(Example \ref{EX-CDP-loop}) The history of the shifts obtained using Algorithm \ref{ALG-HOKIE}.
{The iterations are colored  using six colors periodically as follows: $\textcolor[rgb]{0.00,0.00,1.00}{\bullet}$
$\textcolor[rgb]{0.00,1.00,0.00}{\bullet}$
$\textcolor[rgb]{0.98,0.00,0.00}{\bullet}$
$\textcolor[rgb]{0.00,1.00,1.00}{\bullet}$
$\textcolor[rgb]{1.00,0.00,1.00}{\bullet}$
$\textcolor[rgb]{1.00,1.00,0.00}{\bullet}$
$\textcolor[rgb]{0.00,0.00,1.00}{\bullet}$
$\textcolor[rgb]{0.00,1.00,0.00}{\bullet}$
$\textcolor[rgb]{0.98,0.00,0.00}{\bullet}$
$\textcolor[rgb]{0.00,1.00,1.00}{\bullet}$
$\textcolor[rgb]{1.00,0.00,1.00}{\bullet}$
$\textcolor[rgb]{1.00,1.00,0.00}{\bullet}, \ldots$ Note how the odd and the even iterates build two separated pairs of "smoke rings" (abscissa range $[6000,10000]$); more smaller rings can be identified in the abscissa range $[3000,5000]$. The shifts do not converge to a fixed point, but the difference between the two sub-sequences (the even and the odd indices) converges to a nonzero value.}
\label{FIG_looped-shifts} }
\end{figure}
\end{example}
\noindent Hence, it is of both theoretical and practical interest to explore possibilities for improving the convergence. Supplying good initial shift is certainly beneficial, and in \cite{DGB-VF-Q-1,DGB-VF-Q-2} we show that the less expensive Vector Fitting algorithm can be used for  preprocessing/preconditioning to generate good shifts that are then forwarded to \textsf{IRKA} to advance them to a local optimum.

An alternative course of action is to deploy an additional control in the iterations which will keep steering the shifts toward the desired positions. In fact, an example of such an intervention has  been already used in the numerical implementation of Algorithm \ref{ALG-HOKIE}. Namely, it can happen that in some steps the matrix (\ref{eq::A_r^k+1}) is not stable and some of its eigenvalues (\ref{eq:IRKA-map}) are in the right half-plane. To correct this situation, the unstable ones (real or complex-conjugate pair(s)) are flipped across the imaginary axis, so that the new shifts $\bfsigma^{(k+1)}$ stay in the right-half plane. 

This is an explicit (brute force)  post-festum reassignment of the eigenvalues to correct for stability. In \cite{DGB-VF-Q-1}, we showed that such a step (in the framework of Vector Fitting) can be recast as  pole placement. Now that we have resorted (implicitly) to the pole placement mechanism, we can think of using it as a proactive strategy for improving convergence. 
In the rest of this section, we explore this idea and discover interesting connections with some variations of  \textsf{IRKA}.

\subsection{Reduced input-to-state vector as a pole placement feedback vector}
Motivated by the above discussion,   we reinterpret  (a posteriori) the vector  $\bfq^{(k+1)}$ in \eqref{eq::A_r^k+1} as the feedback vector that reallocates the eigenvalues\footnote{We tacitly assume that throughout the iterations all shifts are simple.} of $\mathrm{diag}(\bfsigma^{(k)})$ into $\bfmu^{(k+1)}$; in other words 
we view \eqref{eq::A_r^k+1} as a pole-placement problem.
Then we can use the uniqueness argument and write 
$\bfq^{(k+1)}$
explicitly as (see \cite{mehrmann1996app}, \cite{mehrmann1998cps}) 
\begin{equation}\label{eq:MXu-q}
-q^{(k+1)}_i = \frac{\prod_{j=1}^r(\sigma^{(k)}_i - \mu^{(k+1)}_j)}{\prod_{\stackrel{j=1}{j\neq i}}^r(\sigma^{(k)}_i - \sigma^{(k)}_j)} = (\sigma^{(k)}_i - \mu^{(k+1)}_i)\prod_{\stackrel{j=1}{j\neq i}}^r \frac{\sigma^{(k)}_i - \mu^{(k+1)}_j}{\sigma^{(k)}_i - \sigma^{(k)}_j} ,\; i = 1, \ldots, r.
\end{equation}
On the other hand, for the fulfillment of
the necessary conditions for optimality, besides the Hermite interpolation  built in (\ref{eq:MXu-q}), the additional fixed point condition should hold:
\begin{equation}\label{eq::reflect_poles}
\mathrm{eig}(\bfA_r^{(k+1)}(\bfsigma^{(k)})) \approx -\bfsigma^{(k)}.
\end{equation}
The latter is what we hope to reach with the equality in
the limit as $k\rightarrow\infty$, and in practice up to a reasonable tolerance, see \S \ref{SS=Limitations-fin-prec}. If we consider the condition (\ref{eq::reflect_poles}) as an eigenvalue assignment problem,
and think of the vector $\bfq^{(k+1)}$ in (\ref{eq::A_r^k+1}) simply as the feedback vector, then
(\ref{eq::reflect_poles}) can be satisfied in one step, provided we drop the interpolation condition
and use an appropriate feedback $\bff^{(k+1)}$ vector instead of $\bfq^{(k+1)}$.
The feedback $\bff^{(k+1)}$ can be constructed explicitly (see \cite{mehrmann1996app,mehrmann1998cps}) as
\begin{equation}\label{eq::f^k+1}
f^{(k+1)}_i = -2\sigma^{(k)}_i\prod_{\stackrel{j=1}{j\neq i}}^{r}
\frac{\sigma^{(k)}_i+\sigma^{(k)}_j}{\sigma^{(k)}_i - \sigma^{(k)}_j}, \;\;i=1,\ldots, r.
\end{equation}
Of course, the above formula is a special case of (\ref{eq:MXu-q}), where we reflect the poles. However, if at a particular step some of the eigenvalues are unstable, we should not reflect the corresponding shifts. This means that in an implementation, we may apply only  partial pole placement.  
Hence, altogether,  it would make sense to interweave interpolation and eigenvalue assignment by combining
$\bff^{(k+1)}$ and $\bfq^{(k+1)}$ using an appropriately chosen parameter
$\alpha_k \in [0,1]$, and thus obtain a modified
iteration step, as outlined in Algorithm \ref{ALG-IRKA+PP}. 
{To incorporate this new shift-updating scheme into \textsf{IRKA}, Algorithm \ref{ALG-IRKA+PP} should replace Step 8 in Algorithm \ref{ALG-HOKIE}.}
\begin{algorithm}[hh]
\caption{\textsf{IRKA} + \textsf{pole placement} for shift updates; $k$th step} \label{ALG-IRKA+PP}
\begin{algorithmic}[1]
\STATE Compute the reduced input vector $\bfq^{(k+1)}$.
\STATE Compute the feedback vector $\bff^{(k+1)}$ using (\ref{eq::f^k+1}). (Keep track of stability.)
\STATE $\breve{\bfq}^{(k+1)} = \alpha_k \bfq^{(k+1)} + (1-\alpha_k)\bff^{(k+1)}$, with an appropriate $\alpha_k\in [0,1]$.
\STATE $\breve{\bfA}_r^{(k+1)}(\bfsigma^{(k)}) = \mathrm{diag}(\sigma_i^{(k)})_{i=1}^r - \breve{\bfq}^{(k+1)}\bfe^T$.
\STATE The new shifts are $\bfsigma^{(k+1)} = - \mathrm{eig}(\breve{\bfA}_r^{(k+1)}(\bfsigma^{(k)}))$.
\end{algorithmic}
\end{algorithm}

\begin{proposition}
For the real  LTI system \eqref{origLTIsys}, and $\bfsigma^{(k)}$ closed under complex conjugation,  the matrix $\breve{\bfA}_r^{(k+1)}(\bfsigma^{(k)})$ is similar to a real matrix and, thus, $\bfsigma^{(k+1)}$ remains closed under complex conjugation.
\end{proposition}
\begin{proof}
From (\ref{eq::f^k+1}), we conclude that $f_i^{(k+1)}$ is real if $\sigma_i^{(k)}$ is real. Further, if for some $i\neq j$ $\sigma_i^{(k)} = \overline{\sigma_j^{(k)}}$, then $f_i^{(k+1)} = \overline{f_j^{(k+1)}}$. We have already concluded that $\bfq^{(k+1)}$ has an analogous structure. Since $\alpha_k$ is real, $\breve{\bfA}_r^{(k+1)}(\bfsigma^{(k)})$ is similar to a real matrix. \hfill$\Box$
\end{proof}

It remains an open problem how to chose the coefficients $\alpha_k$ adaptively and turn Algorithm \ref{ALG-IRKA+PP} into a robust black-box scheme. We now show an interesting connection that might provide some guidelines. 

\subsection{Connection to the Krajewski--Viaro scheme}
Improving the convergence of fixed point iterations is an important topic. 
In general, the fixed point problem $f(x)=x$ can be equivalently solved as
the  problem
$$
f_\alpha(x) = x,\;\;\mbox{where}\;\;f_\alpha(x) = \alpha f(x) + (1-\alpha)x,\;\;\alpha\neq 0,
$$
where the parameter $\alpha$ is used, e.g., to modify eigenvalues of the corresponding Jacobian.
This is a well--known technique (Mann iteration), with many variations.
In the context of $\mathcal{H}_2$ model reduction, this scheme has been successfully
applied by Ferrante, Krajewski, Lepschy and Viarro \cite{Ferrante_Krajewski...:1999:Convergent_alg_L2},
and Krajewski and Viaro \cite{Krajewski_Viaro:2009:Iterative_Interpolation__alg_L2}.
Concretely, Krajewski and Viaro \cite[Algorithm 4.]{Krajewski_Viaro:2009:Iterative_Interpolation__alg_L2} 
propose a modified step for the \textsf{IRKA} procedure, outlined in Algorithm \ref{ALG-KR-VI}:

\begin{algorithm}[hh]
\caption{Krajewski-Viaro scheme for shift updates; $k$th step} \label{ALG-KR-VI}
\begin{algorithmic}[1]
\STATE Let $\breve{\bfwp}_r^{(k)}$ be the modified (monic) polynomial, whose reflected zeros are the
shifts $\bfsigma^{(k)}$, i.e. $\breve{\bfwp}_r^{(k)}(-\sigma^{(k)}_i)=0$, $i=1,\ldots, r$.
\STATE Compute
$\bfA_r^{(k+1)} = (\bfW_k^T \bfV_k)^{-1}\bfW_k^T \bfA \bfV_k$ and the coefficients of its characteristic polynomial $\bfwp_r^{(k+1)}$.
Write this mapping as $\bfwp_r^{(k+1)} = \Phi(\breve{\bfwp}_r^{(k)})$.
\STATE Define the new polynomial
\begin{equation}\label{Krajewski-pk+1}
\breve{\bfwp}_r^{(k+1)} = \alpha \bfwp_r^{(k+1)} + (1-\alpha) \breve{\bfwp}_r^{(k)} \equiv \Phi_{\alpha}(\breve{\bfwp}_r^{(k)}),
\end{equation}
where the linear combination of the polynomials is formed using their coefficients.
\STATE The new shifts are then the reflected roots of $\breve{\bfwp}_r^{(k+1)}$.
\end{algorithmic}
\end{algorithm}

\noindent Krajewski and Viaro \cite{Krajewski_Viaro:2009:Iterative_Interpolation__alg_L2} do not elaborate the details of computing the coefficients of the
characteristic polynomials of the reduced matrix $(\bfW_k^T \bfV_k)^{-1}\bfW_k^T \bfA \bfV_k$. From a numerical point of view, this is not feasible, not even for moderate dimensions. Computing coefficients of the characteristic polynomial by, e.g., the classical Faddeev-Leverrier trace formulas is both too expensive ($O(r^4)$) and too ill-conditioned. A modern approach would reduce the matrix to a Schur or Hessenberg form and then exploit the triangular or, respectively, Hessenberg structure. However, after completing all those (tedious) tasks, the zeros of $\breve{\bfwp}_r^{(k+1)}$ in Line 4 are best computed by transforming the problem into an eigenvalue computation for an appropriate companion matrix.
Ultimately, this approach is only conceptually interesting as a technique for improving convergence, and in this form it is applicable only for small values of $r$. 

We now show that when represented in a proper basis, this computation involving characteristic polynomials becomes rather elegant and simple,
and further provides interesting insights. In particular, in this proper basis, Algorithm \ref{ALG-KR-VI} is equivalent to Algorithm \ref{ALG-IRKA+PP}.
\begin{theorem}
In the Lagrange basis of $\omega_r^{(k)} + \mathcal{P}_{r-1}$, the Krajewski--Viaro iteration is
{equivalent to}
the ``\textsf{IRKA} + pole placement"  iteration of 
Algorithm \ref{ALG-IRKA+PP}.
\end{theorem}
\begin{proof}
Note that by Lemma \ref{RatKryCanonForm}  we can write $\bfA_r^{(k+1)} = \diag(\sigma^{(k)}_i) - \bfq^{(k+1)}\bfe^T$, where 
$q^{(k+1)}_i = \bfwp_r^{(k+1)}(\sigma_i)/(\omega_r^{(k+1)})'(\sigma_i)$, $i=1,\ldots, r$, and $\bfwp_r^{(k+1)}(z)$ is the characteristic polynomial of $\bfA_r^{(k+1)}$.
Further, using Lemma \ref{RatKryCanonForm}, we can write  $\bfwp_r^{(k+1)}(z)$ as 
\begin{eqnarray*}
\bfwp_r^{(k+1)}(z) &=& \omega_r^{(k)}(z) + \sum_{i=1}^r \bfwp_r^{(k+1)}(\sigma^{(k)}_i) \frac{\omega_r^{(k)}(z)}{(\omega_r^{(k)})'(\sigma^{(k)}_i) (z-\sigma^{(k)}_i)}\\
&=& \omega_r^{(k)}(z) + \sum_{i=1}^r q^{(k+1)}_i \frac{\omega_r^{(k)}(z)}{z-\sigma^{(k)}_i},\;\;
\omega_r^{(k)}(z) = \prod_{i=1}^r (z-\sigma^{(k)}_i) .
\end{eqnarray*}
If we consider the monic polynomials of degree $r$ as the linear manifold $\omega _r^{(k)}+ \mathcal{P}_{r-1}$, and
fix in $\mathcal{P}_{r-1}$ the Lagrange basis with the nodes $\sigma^{(k)}_i$, $i=1,\ldots, r$, then
\begin{eqnarray*}
\breve{\bfwp}_r^{(k)}(z) &=& \omega_r^{(k)}(z) + \sum_{i=1}^r \breve{\bfwp}_{k}(\sigma^{(k)}_i) \frac{\omega_r^{(k)}(z)}{(\omega_r^{(k)})'(\sigma^{(k)}_i) (z-\sigma^{(k)}_i)}\\
 &=& \omega_r^{(k)}(z) + \sum_{i=1}^r \frac{\prod_{j=1}^r (\sigma^{(k)}_i + \sigma^{(k)}_j)}{\prod_{\stackrel{j=1}{j\neq i}}^r (\sigma^{(k)}_i - \sigma^{(k)}_j)}  \frac{\omega_r^{(k)}(z)}{z-\sigma^{(k)}_i} = \omega_r^{(k)}(z) + \sum_{i=1}^r f^{(k+1)}_i \frac{\omega_r^{(k)}(z)}{z-\sigma^{(k)}_i},
\end{eqnarray*}
where we used that $\breve{\bfwp}_{r}^{(k)}(\sigma^{(k)}_i )= \prod_{j=1}^r (\sigma^{(k)}_i + \sigma^{(k)}_j)$,  
$(\omega_r^{(k)})'(\sigma^{(k)}_i)=\prod_{\stackrel{j=1}{j\neq i}}^r (\sigma^{(k)}_i - \sigma^{(k)}_j)$ and that the feedback vector (\ref{eq::f^k+1}) can be written as 
$$
\frac{\prod_{j=1}^r (\sigma^{(k)}_i + \sigma^{(k)}_j)}{\prod_{\stackrel{j=1}{j\neq i}}^r (\sigma^{(k)}_i - \sigma^{(k)}_j)} = 2 \sigma^{(k)}_i
\prod_{\stackrel{j=1}{j\neq i}}^r \frac{(\sigma^{(k)}_i + \sigma^{(k)}_j)}{(\sigma^{(k)}_i - \sigma^{(k)}_j)} = f_i^{(k+1)}.
$$
Hence, $\breve{\bfwp}_r^{(k)}$ is the characteristic polynomial of   $\mathrm{diag}(\bfsigma^{(k)}) - \bff^{(k+1)}\bfe^T$, and we have further
\begin{eqnarray}
\breve{\bfwp}_r^{(k+1)}(z) &=& \omega_r^{(k)}(z) + \sum_{i=1}^r (\alpha_k q^{(k+1)}_i + (1-\alpha_k)f^{(k+1)}_i) \frac{\omega_r^{(k)}(z)}{z-\sigma^{(k)}_i} \\
&=& \omega_r^{(k)}(z) + \sum_{i=1}^r \frac{\alpha_k \bfwp_r^{(k+1)}(\sigma^{(k)}_i) + (1-\alpha_k)\breve{\bfwp}_r^{(k)}(\sigma^{(k)}_i)}{(\omega_r^{(k)})'(\sigma^{(k)}_i)} \frac{\omega_r^{(k)}(z)}{z-\sigma^{(k)}_i} \\
&=& \omega_r^{(k)}(z) + \sum_{i=1}^r \frac{\breve{p}_r^{(k+1)}(\sigma^{(k)}_i)}{(\omega_r^{(k)})'(\sigma^{(k)}_i)} \frac{\omega_r^{(k)}(z)}{z-\sigma^{(k)}_i},
\end{eqnarray}
which implies that $\breve{\bfwp}_r^{(k+1)}$ is the characteristic polynomial of the matrix
$\breve{\bfA}_r^{(k+1)}(\bfsigma^{(k)})$, represented by the vector $\breve{\bfq}^{(k+1)}$ from Lines 3 and  4 of Algorithm \ref{ALG-IRKA+PP}. This follows from the proof of Lemma \ref{RatKryCanonForm}. \hfill$\Box$
\end{proof}

\noindent Krajewski and Viaro \cite{Krajewski_Viaro:2009:Iterative_Interpolation__alg_L2} use a fixed value of the parameter $\alpha$, and show that different (fixed) values
may lead to quite different convergence behavior. This modification can turn a non-convergent process into a convergent one,
but it can also slow down already convergent one.\footnote{We should point out here that the dimensions $n$ and $r$ are
rather small in all reported numerical experiments in \cite{Krajewski_Viaro:2009:Iterative_Interpolation__alg_L2}.}
Following the discussion from
\cite{Ferrante_Krajewski...:1999:Convergent_alg_L2}, $\alpha$ is best chosen to move the smallest
eigenvalue of the Jacobian of $\Phi_\alpha$ (evaluated at the fixed point of $\Phi$) into the interval $(-1,1)$.
This does not seem to be a simple task as it requires estimates of the eigenvalues of the (estimated)
Jacobian. Another option is to try different values of $\alpha$ in an iterative reduced order model design.

This equivalence of the schemes in Algorithm \ref{ALG-IRKA+PP} and in Algorithm \ref{ALG-KR-VI} reveals a problem that is not easily seen in the framework of Algorithm \ref{ALG-KR-VI}.
Now we may clearly see that part of the ``energy'' in Algorithm \ref{ALG-KR-VI} is put into reflecting the
shifts, and this, at least in some situations, may be wasted effort.
Although the optimal reduced order model is guaranteed
to be stable, the iterates, generally, are not. This means that some shifts $\sigma^{(k)}_i$ may be in the left--half plane, and the
$\bff^{(k+1)}$ component of the modified $\breve{\bfq}^{(k+1)}$ will tend to reflect them to the right--half plane. This in turn
forces the new reduced matrix $\breve{\bfA}^{(k+1)}_r$ to have some eigenvalues in the right--half plane, thus creating a vicious circle. Hence, in the first step (Line 1), one should correct $\bfsigma^{(k)}$, if necessary.

\begin{remark}
The facts that pole placement may be extremely ill-conditioned \cite{he:ppp}, where (depending on the distribution of the target values) \emph{plenty} may be even as small as $r=15$,  and that \textsf{IRKA} is actually doing  pole placement in disguise, open many nontrivial issues. For instance, what is a reasonable threshold for the stopping criterion? Will we be able to actually test it (and detect convergence) in finite precision computation? What are relevant condition numbers in the overall process? Do \textsf{IRKA} iterations drive the shifts to well-conditioned configurations for which the feedback vector (the reduced input) is reasonably small 
(in norm, as illustrated in the right panels of Figure \ref{FIG_CDcondC} and Figure \ref{FIG_ISS1RcondC} in Example \ref{example:condC} below)
and successful in achieving numerical convergence?  If yes, what is the underlying principle/driving mechanism?
In the next section, we touch upon some of these issues.
\end{remark}

\section{Perturbation effects and backward stability in \textsf{IRKA}}
\label{sec:backstab}

We turn our attention now to numerical considerations in the implementation of IRKA. We focus on two issues: (i) What are the perturbative effects of finite-precision  arithmetic in terms of system-theoretic quantities? (ii) What are the effects of ``numerical convergence" on the reduced model?

\subsection{Limitations of finite precision arithmetic}\label{SS=Limitations-fin-prec}

Suppose that we are  given magic shifts $\bfsigma$ so that the eigenvalues of
$\bfA_r = \bfSigma_r - \bfq \bfe^T$ are exactly $\bflambda = -\bfsigma$,
or $\bflambda \approx -\bfsigma$ up to a small tolerance $\bfeps$.  However, in
floating point computations, the vector $\bfq = (\bfW^T \bfV)^{-1}\bfW^T \bfb$ is computed up to an
error $\bfdq$, and therefore instead of $\bfA_r$, we have
$\widetilde{\bfA}_r=\bfSigma_r-({\bfq}+
\bfdq)\bfe^T$.
In practice, the source of $\bfdq$ is twofold: First, in large-scale settings,
the primitive Krylov bases $\bfV$ and $\bfW$ are usually computed by an iterative method which uses
restricted information from suitably chosen subspaces and thus generates
a truncation error; see, e.g., \cite{beattie2012inexact,ahuja2012recycling,ahuja2015recycling}. In addition, computation is polluted by omnipresent rounding errors of finite precision arithmetic. How the size of $\bfdq$ influences the other
components of the \textsf{IRKA} is a relevant information  that we investigate in this subsection.

Assume for the moment that
$\bfdq$ is the only perturbation in one step of \textsf{IRKA}. We want to understand how the eigenvalues
of $\bfAr$ and $\tbfAr$ differ as a function of $\bfdq$. In particular, we want to reveal the relevant condition numbers that play a role in this perturbation analysis.

\begin{theorem}  \label{mumutildethm}
Let $\bfAr = \bfSigma_r - \bfq\bfe^T$ and
$\tbfAr = \bfSigma_r - \tbfq\bfe^T$ be
diagonalizable, where
$\tbfq={\bfq}+\bfdq$. Let 
$\bfAr$ and $\tbfAr$ have the
 spectral
decompositions  $\bfAr = \bfX \bfM
{\bfX}^{-1}$ and $\tbfAr = \tbfX
\tbfM \tbfX^{-1}$, where
$\bfM=\mathrm{diag}(\mu_i)_{i=1}^r$,
$\tbfM=\mathrm{diag}(\tilde{\mu}_i)_{i=1}^r$ and the
eigenvector matrices
${\bfX}=\bfD_{\bfq}{\bfC}$ and
$\tbfX=\tbfD_\bfq\tbfC$ as
described in Corollary \ref{EVDArArtCor}.
Then there exists a permutation $\pi$ such that
\begin{equation}\label{eq:Wieland-Hoffman}
\sqrt{\sum_{i=1}^r \left|\frac{\mu_i -
\tilde\mu_{\pi(i)}}{\mu_i}\right|^2} \leq \left\|{\bfC}\right\|_2
\left\|({\bfC}\bfM)^{-1}\right\|_2 \kappa_2(\tbfC)
\left\|\bfdq\bfe^T\right\|_2.
\end{equation}
\end{theorem}
\begin{proof}
Note that we can equivalently compare the spectra of
$\bfAr^T$ and $\tbfAr^T$. From \eqref{leftRitzvector}, we know that the spectral decomposition of $\bfAr^T$ is given by
\begin{equation} \label{EVDArt-copy}
\bfAr^T =  \bfSigma_r - \bfe\bfq^T = 
\bfD_{{\bfq}}^{-1} \bfAr \bfD_{{\bfq}} = 
\bfC \bfM \bfC^{-1}.
\end{equation}
Similarly for $\tbfAr^T$, we obtain
\begin{equation}\label{EVDtildeArt}
\tbfAr^T = 
\bfSigma_r - \bfe\tbfq^T = 
\tbfD_{\bfq}^{-1}
\tbfAr \tbfD_{\bfq}= \tbfC \tbfM \tbfC^{-1}.
\end{equation}
Next we employ the perturbation results of Elsner and
Friedland \cite{elsner1995svd}, and Eisenstat and Ipsen \cite{eisenstat1998tap}, while taking into
account the special structure of both matrices. Write $\tbfAr^T =
\bfAr^T + \bfdelta \bfAr^T$, where
$\bfdelta\bfAr^T=-\bfe\bfdelta\bfq^T$. Using the
spectral decompositions \eqref{EVDArt} and \eqref{EVDtildeArt}, the matrix
$\bfAr^{-T}\tbfAr^T -
\bfI=\bfAr^{-T}\bfdelta\bfAr^T$ can be
transformed into
\begin{equation}\label{eq:Y}
\bfM^{-1}
({\bfC}^{-1}\tbfC)\tilde{\bfM} -
({\bfC}^{-1}\tbfC) =
{\bfC}^{-1}\bfAr^{-T}\bfdelta \bfAr^T\tbfC.
\end{equation}
Set ${\bfY}={\bfC}^{-1}\tbfC$ and take the
absolute value of ${y}_{ij}$, an arbitrary entry of $\bfY$ at the $(i,j)$th position, to obtain
$$
|{y}_{ij}| \left| \frac{\tilde{\mu}_j}{\mu_i} - 1\right| =
\left|
({\bfC}^{-1}\bfAr^{-T}\bfdelta \bfAr^T\tbfC)_{ij}\right|\;\;\mbox{for}~~1\leq
i,j \leq r.
$$
Hence $${\displaystyle \sum_{i=1}^r\sum_{j=1}^r|{y}_{ij}|^2
\left| \frac{\tilde{\mu}_j}{\mu_i} - 1\right|^2 = \|
{\bfC}^{-1}\bfAr^{-T}\bfdelta \bfAr^T\tbfC
\|_F^2, }
$$ where the Hadamard product matrix ${\bfY}\circ
\overline{{\bfY}}=(|{y}_{ij}|^2)_{i,j=1}^r$ is
entry--wise bounded by 
$$
\sigma_{\min}(\bfY)^2\; {\bfS}_{ij} \leq ({\bfY}\circ
\overline{{\bfY}})_{ij} \leq \sigma_{\max}({\bfY})^2\;
{\bfS}_{ij},
$$
where $\bfS$ is a doubly--stochastic matrix; see  \cite{elsner1995svd}.
Hence
\begin{equation}\label{eq:Elsner+Friedland}
\sum_{i=1}^r\sum_{j=1}^r {\bfS}_{ij} \left|
\frac{\tilde{\mu}_j}{\mu_i} - 1\right|^2 \leq
\|{\bfY}^{-1}\|_2^2\|
{\bfC}^{-1}\bfAr^{-T}\bfdelta \bfAr^T\tbfC
\|_F^2 \leq \kappa_2({\bfC})^2 \kappa_2(\tbfC)^2
\|\bfAr^{-T}\bfdelta \bfAr^T\|_F^2.
\end{equation}
The expression on the left--hand side of (\ref{eq:Elsner+Friedland})
can be considered as a function defined on the convex polyhedral set
of doubly--stochastic matrices, whose extreme points are the
permutation matrices. Thus, for some permutation $\pi$, we obtain
$$
\sum_{i=1}^r \left|
\frac{\tilde{\mu}_{\pi(i)}-\mu_i}{\mu_i}\right|^2 \leq
\|{\bfY}^{-1}\|_2^2\|
{\bfC}^{-1}\bfAr^{-T}\bfdelta \bfAr^T\tbfC
\|_F^2.
$$
Then, using \eqref{EVDArt}, the spectral decomposition of $\bfAr^{-T}$, and the
definition of $\bfdelta \bfAr$ complete the proof. 
\hfill$\Box$
\end{proof}
\begin{remark}
One can write (\ref{eq:Y}) as $ {\mathbf{Y}}\tbfM -
\bfM{\mathbf{Y}} =
\mathbf{\bfC}^{-1}\bfdA_r^T{\tbfC} $ and conclude
that there exists a permutation $p$ such that (see
\cite{elsner1995svd})
$$
\sqrt{\sum_{i=1}^n |\tilde{\mu}_{p(i)} - \mu_i|^2} \leq
\kappa_2(\mathbf{C})\kappa_2(\tilde{\mathbf{C}})\|\bfdq\bfe^T\|_2.
$$
\end{remark}

\begin{remark}
The right--hand side in relation \eqref{eq:Wieland-Hoffman} can also be
bounded by
$$\sqrt{r}\kappa_2({C})\kappa_2(\tbfC)\frac{\|\bfdq\|_2}{\min_{i}|\mu_i|}.$$
\end{remark}

From the numerical point of view, Theorem \ref{mumutildethm} cannot be good news --
Cauchy matrices can be ill--conditioned. A few random trials will quickly produce
a $10\times 10$ Cauchy matrix with condition number greater than $10^{10}$. 
The most notorious example is the Hilbert matrix, which at the dimension $100$ has a condition number larger than $10^{150}$. 
No function of the matrix
that is influenced by that condition number of the eigenvectors can
be satisfactorily computed in $32$ bit machine arithmetic. The
$64$ bit double precision allows only slightly larger dimensions
before the condition number takes over the machine double precision.

On the other hand, we note that our goal is not to place the shifts
at any predefined locations in the complex plane. Instead, we are
willing to let them go wherever they want, under the condition
that they remain closed under conjugation and stationary at those
positions. It should also be noted that the distribution of the
shifts obviously plays a role in this considerations. The following example
will illustrate this; especially the impact of the optimal $\mathcal{H}_2$ interpolation points.

\begin{example}
\label{example:condC}
As in Example \ref{EX-CDP-loop},  we first take the CD player model \cite{ChaV2005,DraSB92} of order $n=120$ and apply \textsf{IRKA} as in Algorithm \ref{ALG-HOKIE} for $r=2$, $r=16$, and $r=26$. In each case, \textsf{IRKA}  is initialized by randomly assigned shifts. The condition numbers of $\bfC$ for each case, recorded throughout the iterations, are shown on the left panel of Figure \ref{FIG_CDcondC}. \textsf{IRKA} drastically reduces the condition number of $\bfC$  throughout the iteration, more than $15$ order of magnitudes for $r=16$
and $r=26$ cases. Therefore,  \textsf{IRKA} keeps assigning shifts in such a way that $\bfC$ becomes better and better conditioned; thus in affect limiting the perturbation effects predicted by Theorem \ref{mumutildethm}. Moreover, we can observe that the reduced input vectors $\bfq^{(k)}$, which act also as feedback vectors that steer the shifts, diminish in norm over the iterations; see the right panel of Figure \ref{FIG_CDcondC}.

\begin{figure}[hh]
\begin{center}
\includegraphics[scale=0.38]{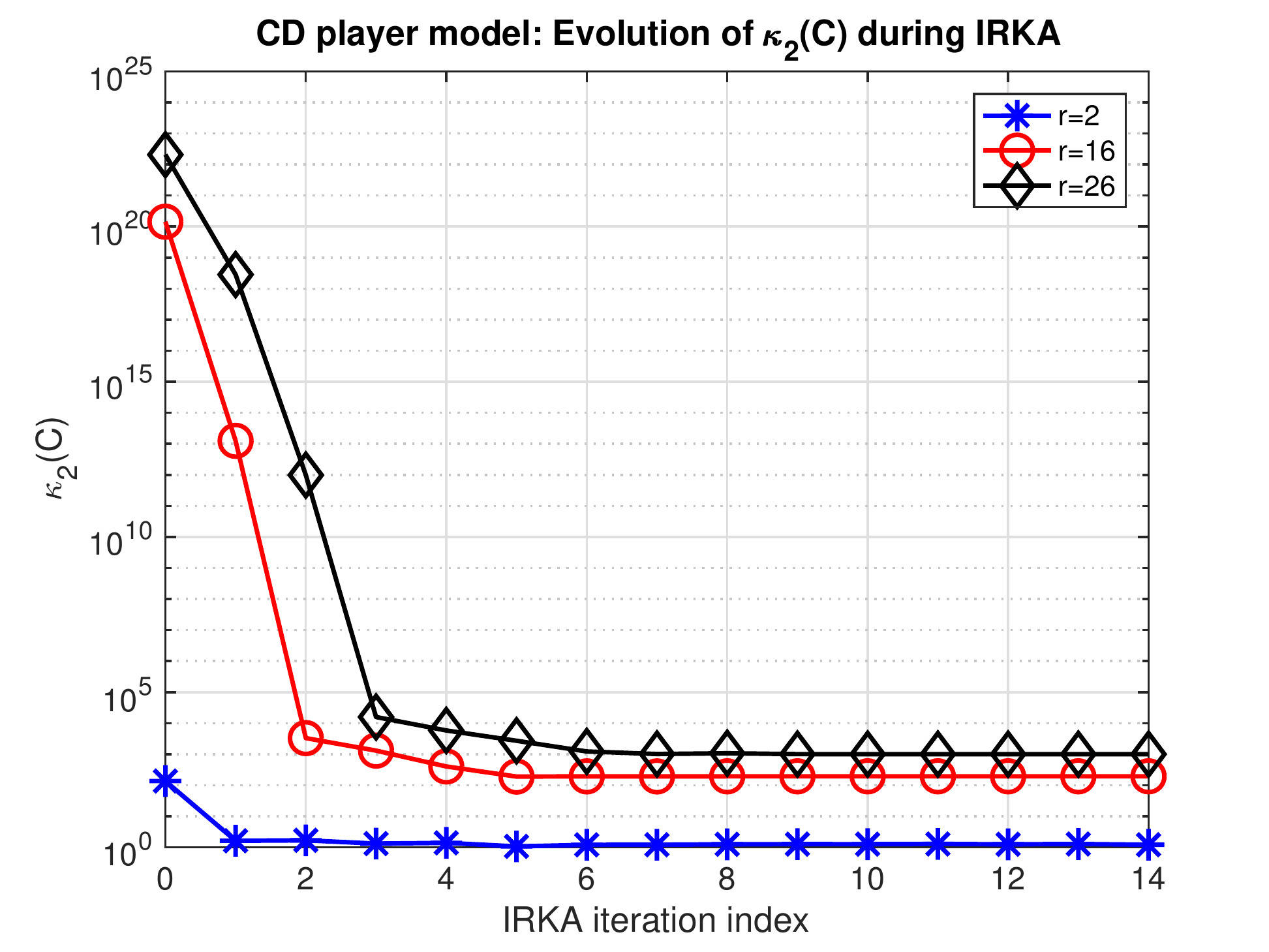}
\includegraphics[scale=0.38]{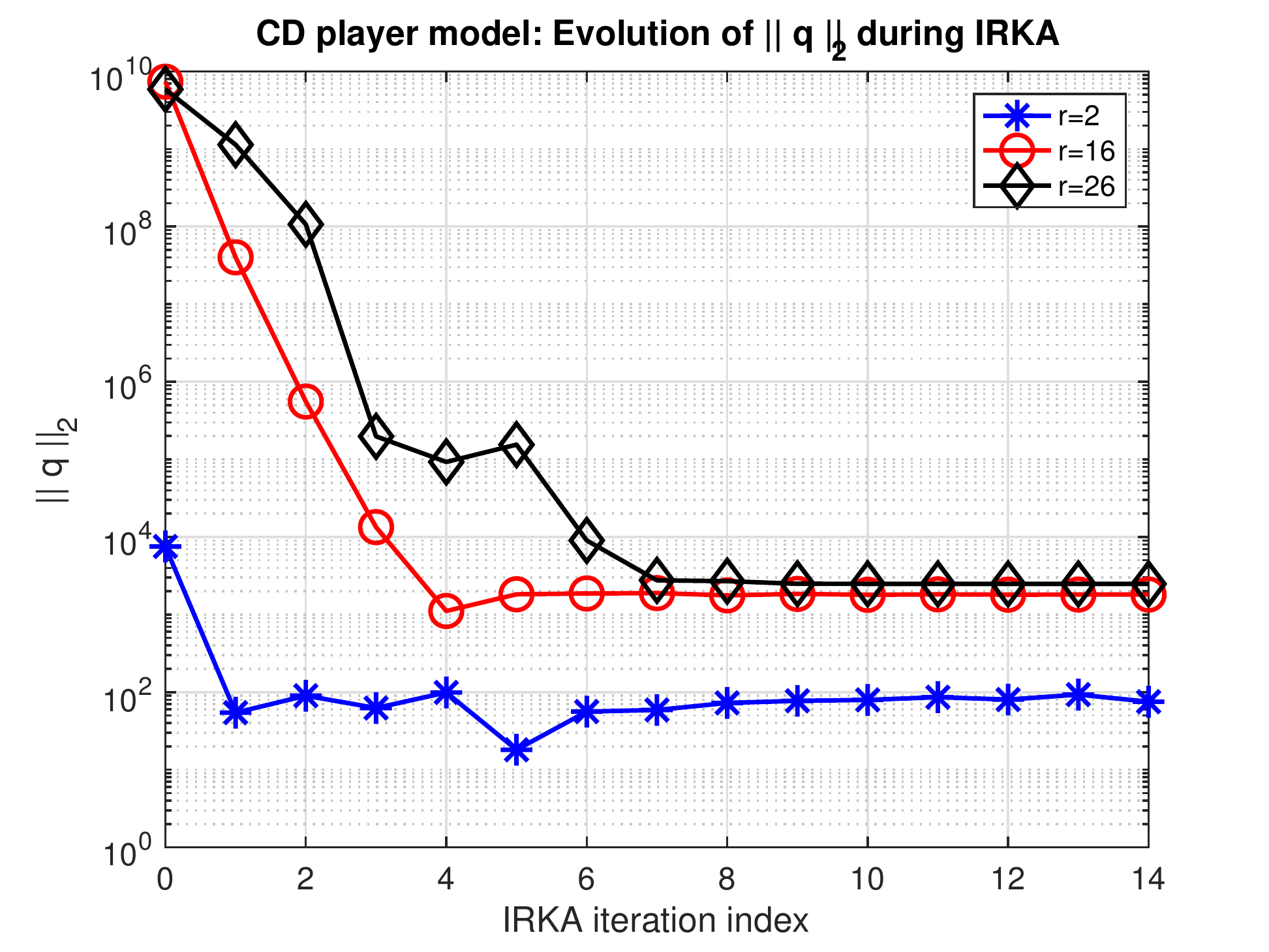}
\end{center}
\caption{$\kappa_2(\bfC)$ and $\|\bfq^{(k)}\|_2$ during \textsf{IRKA} for the CD Player example}
\label{FIG_CDcondC}
\end{figure}

We have observed the same effect in all the examples we have tried. For brevity, we include only one more such result using the International Space Station 1R Module \cite{morGugAB01,antoulas2001asurvey} of order $n=270$. As for the CD Player model, we reduce this model with \textsf{IRKA} using random initial shifts and this time chose reduced orders of $r=10$, $r=20$, and $r=30$. The results depicted in Figure  
\ref{FIG_ISS1RcondC} reveal the same behavior: The condition number $\kappa_2(\bfC)$ is reduced significantly during 
\textsf{IRKA} as shifts converge to the optimal shifts; the same holds for the reduced input norms $\|\bfq^{(k)}\|_2$. These observation raises intriguing theoretical questions about the distribution of the $\mathcal{H}_2$-optimal shifts, as their impact mimics that of Chebyshev points (as opposed to linearly spaced ones) in polynomial interpolation. These issues will not be studied here and are left to future papers.
\begin{figure}[hh]
\begin{center}
\includegraphics[scale=0.38]{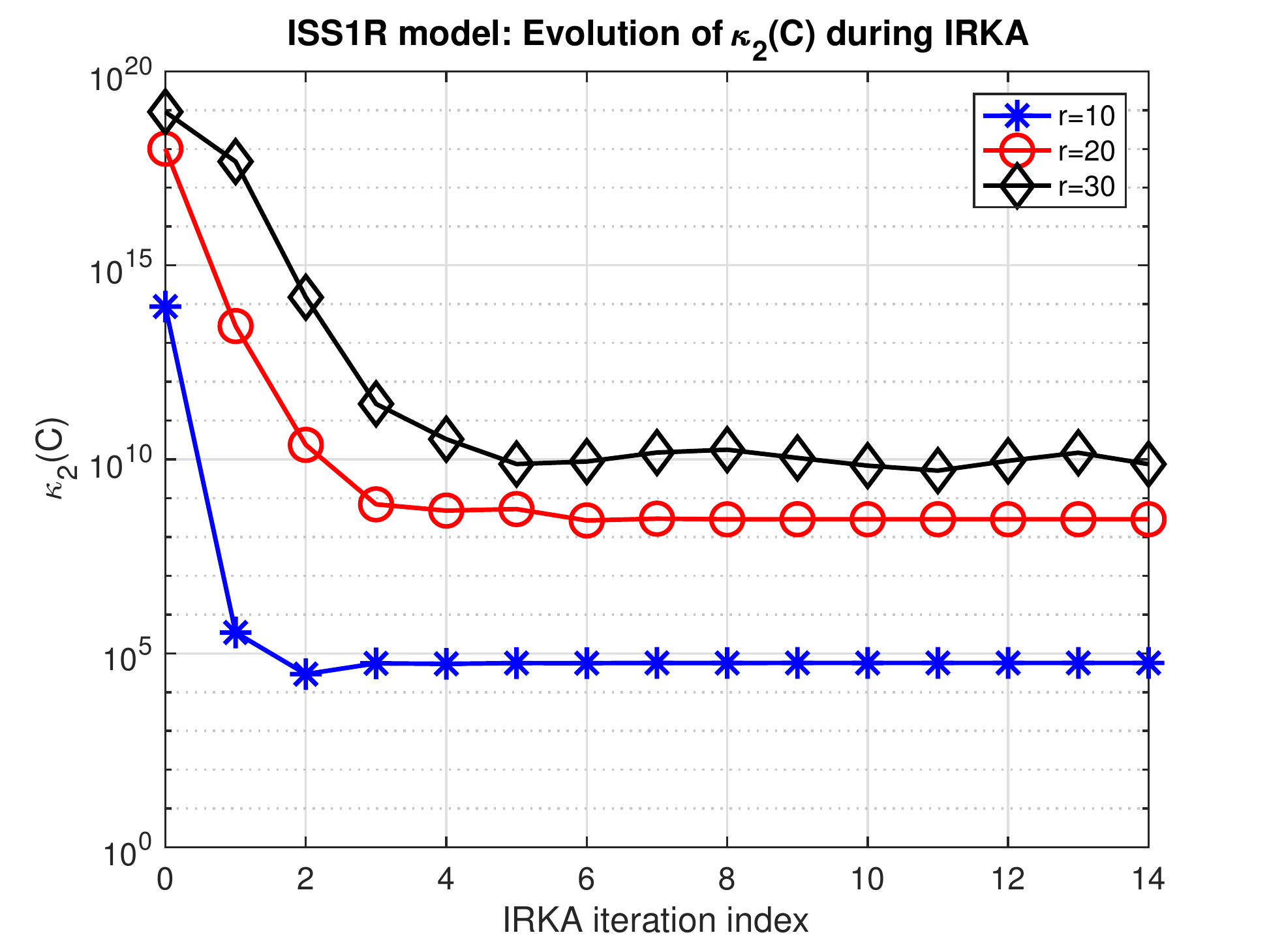}
\includegraphics[scale=0.38]{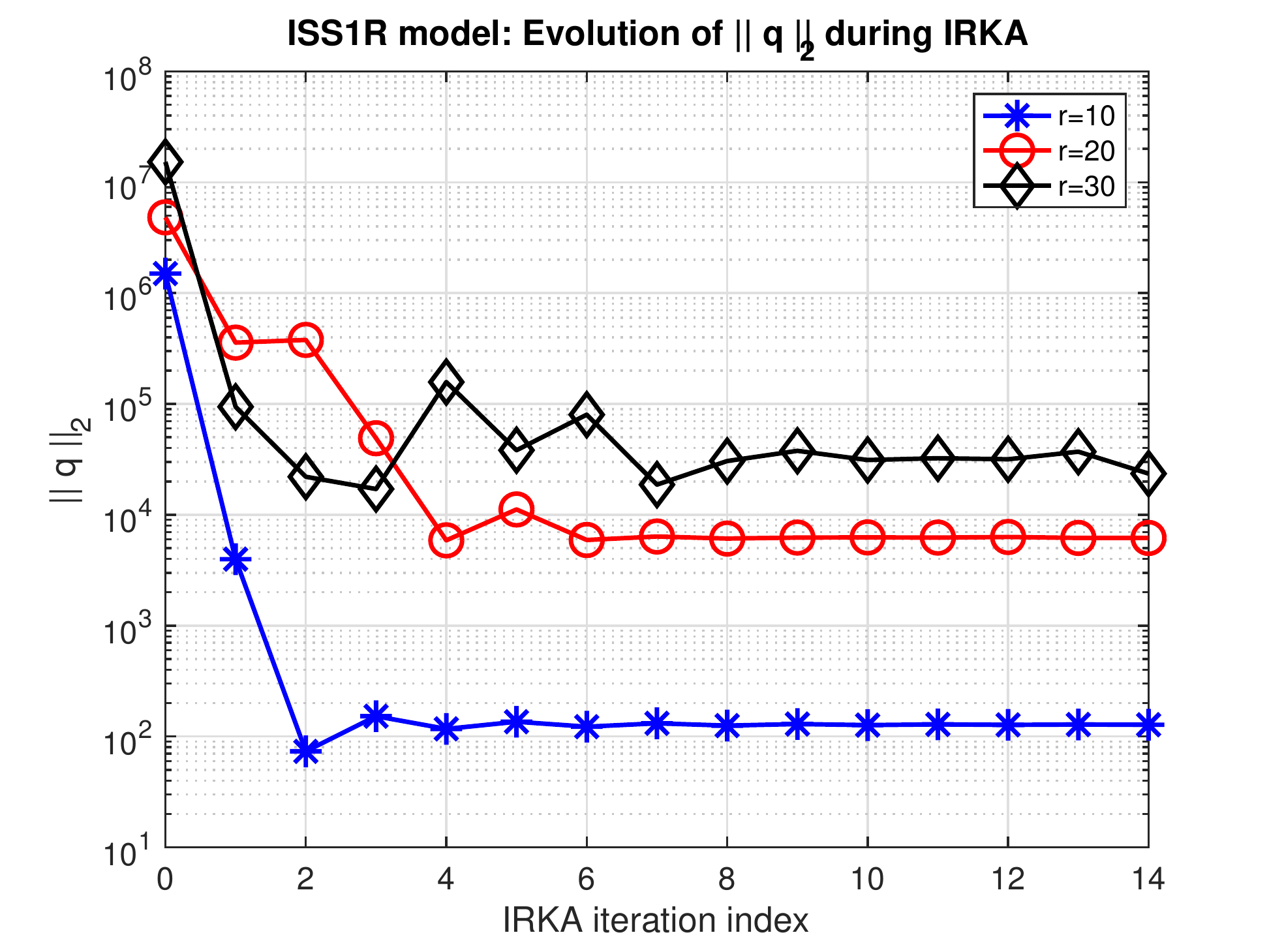}
\end{center}
\caption{$\kappa_2(\bfC)$ and $\|\bfq^{(k)}\|_2$ during \textsf{IRKA} for the IRR 1R example}
\label{FIG_ISS1RcondC}
\end{figure}

\end{example}

\subsection{Stopping criterion and backward stability}
\label{sec:backstab-stop-crit}

Analytically, $\mathcal{H}_2$ optimality is satisfied when $\bfsigma = -\bfmu$. However, in practice Algorithm \ref{ALG-HOKIE} will be terminated once a numerical convergence threshold is met. In this section, we will investigate the impact of numerical convergence on the resulting reduced model. The pole-placement connection we established in Section \ref{sec:poleplacement} will play a fundamental role in answering this question.

Suppose that in Algorithm \ref{ALG-HOKIE} a numerical stopping (convergence) criterion has
been satisfied, i.e., the eigenvalues $\mu_1,\ldots, \mu_k$
of $\bfAr$ \emph{are close to the reflected shifts}.  Both
the shifts $\bfsigma$ and the computed eigenvalues $\bfmu$ are
unordered $r$--tuples of complex numbers, and we measure their distance using optimal matching, see \S \ref{SS=Matching+example}.
Hence, we define the indexing of $\bfmu=(\mu_1,\ldots,
\mu_r)$ so that
$$
\|\bfmu - (-\bfsigma)\|_\infty =
\min_{\pi\in\mathbb{S}_r}\max_{k=1:r}|\mu_{\pi(k)} -
(-\sigma_k)| .
$$
Recall that the shifts $\bfsigma = \{\sigma_1,\ldots,\sigma_r\}$ are closed under
conjugation, with strictly positive real parts, and all assumed to
be simple. The eigenvalues $\bfmu = \{\mu_1,\ldots, \mu_r\}$ are assumed
also simple, and they are obviously closed under
conjugation. With this setup, we
write
\begin{equation}\label{eq:mu_k}
\mu_k =  - \sigma_k + \varepsilon_k,\; k=1,\ldots, r,
\end{equation}
where we note that the $\varepsilon_k$'s are closed under
conjugation as well. Our goal is to relate the $\varepsilon_k$'s and
the quality of the computed reduced order model identified by the triplet $(\bfAr,\bfbr,\bfcr)$ in the sense of backward error. In
particular, we need a yardstick to determine when an $\varepsilon_k$ is
small.

There is a caveat here: for given shifts $\bfsigma$, the vector $\bfmu$ consists of the computed
eigenvalues, thus possibly highly ill--conditioned and computed with large errors.
Our analysis here considers the computed eigenvalues as exact but for a slightly changed input data,\footnote{This is the classical backward error interpretation.} and
we focus on the stopping criterion and how to justify it through a backward
stability statement. This means that we want to interpret the computed reduced order model as an exact reduced order model for an LTI close to the original one (\ref{origLTIsys}).

The representation of the reduced order model in the primitive basis presented in Section \ref{sec:prim_basis} 
yields an elegant structure and provides theoretical insights. On
the other hand, having numerical computation in mind, the same
structure gives reasons to exercise caution. This caution is
particularly justified because, as we showed in Section \ref{sec:poleplacement},
the ultimate step of the iteration is an eigenvalue assignment problem:
find the shifts $\bfsigma$ such that the eigenvalues of
$\bfSigma_r - {\bfq}(\bfsigma)\bfe^T$ are the
reflected shifts.

Using \eqref{eq:MXu-q}, \eqref{eq:mu_k}, and Lemma \ref{RatKryCanonForm}, we know that the vector
$\bfq = \begin{bmatrix} q_1 & q_2 & \ldots & q_r \end{bmatrix}^T$  satisfies
\begin{equation}\label{eq:q_i}
{q}_i = (\sigma_i - \mu_i) \prod_{\stackrel{k=1}{k\neq
i}}^{r} \frac{\sigma_i-\mu_k}{\sigma_i-\sigma_k} =
(2\sigma_i-\varepsilon_i)\prod_{\stackrel{k=1}{k\neq i}}^{r}
\frac{\sigma_i+\sigma_k-\varepsilon_k}{\sigma_i-\sigma_k},\;\;i=1,\ldots,
r.
\end{equation}
We now do the following \emph{Gedankenexperiment}.  Consider the true
reflections of the shifts, $\mu_i^\bullet = - \sigma_i$. Since the pair $(\bfSigma_r,
\bfe)$ is controllable\footnote{Let $\bfA \in \Cplx^{n\times n}$ and 
$\bfb\in \Cplx^n$. Then, the pair $(\bfA,\bfb)$ is called controllable if $\mathrm{rank}\begin{bmatrix}\bfb & \bfA \bfb & \cdots & \bfA^{n-1}\bfb \end{bmatrix}=n.$}, there exists ${\bfq}^\bullet$ such
that the eigenvalues of $\bfSigma_r - {\bfq}^\bullet \bfe^T$
are precisely $\mu_1^\bullet, \ldots, \mu_r^\bullet$. In
fact, by the formula (\ref{eq::f^k+1}), the feedback ${\bfq}^\bullet$ is explicitly given  as
\begin{equation}\label{eq:q^bullet_i}
{q}_i^\bullet = 2\sigma_i\prod_{\stackrel{k=1}{k\neq i}}^{r}
\frac{\sigma_i+\sigma_k}{\sigma_i - \sigma_k}, \;\;i=1,\ldots, r.
\end{equation}
Comparing (\ref{eq:q_i}) and (\ref{eq:q^bullet_i}), we see that our
computed vector ${\bfq}$ satisfies
\begin{equation}
{q}_i = {q}_i^\bullet \prod_{k=1}^r (1 -
\frac{\varepsilon_k}{\sigma_i+\sigma_k})\equiv {q}_i^\bullet
(1-\eta_i), \;\;i=1,\ldots, r.
\end{equation}
Since all the shifts are in the right--half--plane,
$|\sigma_i+\sigma_k|$ is bounded from below by $2\min_j
\mathrm{Re}(\sigma_j)>0$. We find it desirable to have our actually
computed ${\bfq}$ close to ${\bfq}^\bullet$, because
${\bfq}^\bullet$ does exactly what we would like the computed
${\bfq}$ to achieve in the limit. This indicates one possible
stopping criterion for the iterations --  the maximal allowed
distance between the new and the old shifts should guarantee small
$|\varepsilon_k/(\sigma_i+\sigma_k)|$ for all $i, k$.

If we define $\bfAr^\bullet \equiv \bfSigma_r -
{\bfq}^\bullet \bfe^T$, then its eigenvalues  are the
reflections of the shifts. Comparing this outcome with the actually
computed $\bfAr=\bfSigma_r - {\bfq}\bfe^T$, we obtain
the following result.

\begin{proposition}\label{PROP:A_r^bullet}
Let   Algorithm \ref{ALG-HOKIE} be stopped with computed $\bfAr =
\bfSigma_r - {\bfq}\bfe^T$, and let the eigenvalues of
$\bfAr$ be $\mu_k =  - \sigma_k + \varepsilon_k$,
$k=1,\ldots, r$. Let
\begin{equation}
\bfvareps^\bullet \equiv  \max_{1\leq i\leq r} \left| \prod_{k=1}^r
\left( 1 - \frac{\varepsilon_k}{\sigma_i + \sigma_k}\right) - 1
\right| < 1,\;\; \bfvareps \equiv  \max_{1\leq i\leq r} \left|
\prod_{k=1}^r \left( 1 + \frac{\varepsilon_k}{\sigma_i -
\mu_k}\right) - 1 \right| < 1.
\end{equation}
Then there exists $\bfAr^\bullet =
\bfAr+\bfdelta\bfAr$ with eigenvalues
$-\sigma_1,\ldots,-\sigma_r$, and $\bfq^\bullet =
{\bfq}-\bfdq$ such that $\bfAr^\bullet =
\bfSigma_r - {\bfq}^\bullet \bfe^T$;
$\|\bfdq\|_2 \leq \bfvareps \|{\bfq}\|_2$,
$\|\bfdq\|_2 \leq \bfvareps^\bullet
\|{q}^\bullet\|_2$; and
\begin{equation}
\|\bfdelta\bfAr\|_2 \leq  2 \bfvareps^\bullet
\|\bfAr^\bullet\|_2,\;\; \|\bfdelta\bfAr\|_2 \leq
\bfvareps (\|\bfAr\|_2 + (1+\max_{k}\left|
\frac{\varepsilon_k}{\mu_k}\right|)\|\bfAr\|_2).
\end{equation}
\end{proposition}
\begin{proof}
Define ${\bfq}^\bullet$ using \eqref{eq:q^bullet_i} and write
${\bfq}={\bfq}^\bullet +\bfdq$.
Write the actually computed reduced matrix $\bfA_r = \bfSigma_r -
{\bfq} \bfe^T$ as
\begin{equation}\label{eq:A^bullet}
\bfAr = \bfSigma_r - {\bfq}^\bullet \bfe^T - \bfdq \bfe^T\;\;\mbox{or}\;\; \bfAr^\bullet =
\bfSigma_r - {\bfq}^\bullet
\bfe^T,\;\;\mbox{where}\;\;\bfAr^\bullet = \bfAr
+ \bfdq \bfe^T .
\end{equation}
 Note that
$\|\bfSigma_r\|_2 = \mathrm{spr}(\bfAr^\bullet) \leq
\|\bfAr^\bullet\|_2$. Further, using $\bfq^\bullet \bfe^T =
\bfSigma_r - \bfAr^\bullet$ and taking the norm we get
$$
\sqrt{r}\|{\bfq}^\bullet\|_2 \leq
\mathrm{spr}(\bfAr^\bullet) + \|\bfAr^\bullet\|_2 \leq
2 \|\bfAr^\bullet\|_2,
$$
and thus the norm of $\bfdelta\bfAr=\bfdq\bfe^T$ can be
estimated as $$\|\bfdelta\bfAr\|_2 = \sqrt{r} 
\|\bfdq\|_2 \leq \sqrt{r} \bfvareps^\bullet \|{\bfq}^\bullet
\|_2 \leq 2 \bfvareps^\bullet \|\bfAr^\bullet\|_2,$$ 
completing the proof. \hfill $\Box$
\end{proof}
\begin{remark}
We conclude that in the vicinity of our computed data (reduced quantities)
$\bfAr$ and $\bfq$, there exist $\bfAr^\bullet$ and
$\bfq^\bullet$ that satisfy the stopping criterion exactly.
Both $\|\bfAr-\bfAr^\bullet\|_2$ and
$\|\bfq-\bfq^\bullet\|_2$ are estimated by the size of
$\|\bfdq\|_2$. But
there is a subtlety here: we cannot use $\|\bfdq\|_2$ as
the stopping criterion. In other words, if we compute $\bfq$
and conclude that $\|\bfq-\bfq^\bullet\|_2$ is small, it
does not mean that the $\mu_k$'s are close to the reflections of
the $\sigma_k$'s. There is difference between continuity and forward stability.
\end{remark}

Our next goal is to interpret $\bfdAr$ and
$\bfdq$ as the results of backward perturbations in the
initial data $\bfA$, $\bfb$.

\begin{theorem}\label{TM:dA_for_A^bullet}
Under the assumptions of Proposition \ref{PROP:A_r^bullet}, there
exist backward perturbations $\bfdA$ and
$\bfdb$ such that the reduced order system 
$$
\begin{array}{l|l} {\displaystyle \mathbf{A}_r^\bullet=\bfSigma_r-\mathbf{q}^\bullet
\bfe^T =
(\mathbf{W}^T\mathbf{V})^{-1}\mathbf{W}^T(\mathbf{A}+\bfdA)\mathbf{V}}
& {\displaystyle \mathbf{b}_r^\bullet \equiv
\mathbf{q}^\bullet=(\mathbf{W}^T\mathbf{V})^{-1}\mathbf{W}^T(\mathbf{b}-\delta
\mathbf{b})} \\\hline {\displaystyle \mathbf{c}_r^\bullet =
\mathbf{c}_r} & \end{array}
$$
corresponds to exact model reduction of  
the perturbed full-order model described by the triplet of matrices
$(\mathbf{A}+\bfdA,\mathbf{b}-\bfdb,\mathbf{c})$ and
has its poles at the reflected shifts. Let 
$G_r^\bullet(s)=\mathbf{c}_r^T
(s\bfI_r-\mathbf{A}_r^\bullet)^{-1}\mathbf{b}_r^\bullet$ and
$G^\bullet (s) = \mathbf{c}^T (s\bfI_n -
(\mathbf{A}+\bfdA))^{-1}(\mathbf{b}-\bfdb)$
denote the transfer functions of this reduced order system, and the
backward perturbed original system, respectively. Then,
$G_r^\bullet(\sigma_i)=G^\bullet(\sigma_i)$, $i=1,\ldots,r$. The backward perturbations satisfy
\begin{equation}
{\displaystyle \|\bfdb\|_2 \leq
\frac{\kappa_2(\mathbf{V})}{\cos\angle(\mathcal{V,W})} \bfvareps
\|\mathbf{b}\|_2}
\end{equation}
and 
\begin{equation}
{\displaystyle \|\bfdelta \mathbf{A}\|_2 \leq \frac{\kappa_2(\mathbf{V})^2}{\cos\angle(\mathcal{V,W})}
\frac{2\bfvareps^\bullet}{1-2\bfvareps^\bullet}\|\mathbf{A}\|_2},~
\mbox{provided~that}~\bfvareps^\bullet < 1/2
\end{equation}
where $\mathcal{V} = \textsf{Range}(\bfV)$
and $\mathcal{W} = \textsf{Range}(\bfW)$.
\end{theorem}
\begin{proof}
First, recall that $\mathbf{q}=\mathbf{U}^T \mathbf{b}$,
 where $\mathbf{U}^T = (\mathbf{W}^T\mathbf{V})^{-1}\mathbf{W}^T$.
Since $\mathbf{U}^T$ has full row--rank, we can determine
$\bfdb$ such that $\bfdq=\mathbf{U}^T
\bfdb$. 
(The unique $\bfdb$ of minimal
Euclidean norm is
$\bfdb=(\mathbf{U}^T)^\dagger\mathbf{\bfdq}\in\mathcal{W}$.)
Using (\ref{eq:A^bullet}) and $\mathbf{A}_r=\mathbf{U}^T
\mathbf{A}\mathbf{V}$ we can write then
\begin{equation}
\mathbf{U}^T \mathbf{A}\mathbf{V} + \mathbf{U}^T
\bfdb\bfe^T = \bfSigma_r - \mathbf{U}^T
(\mathbf{b}-\bfdb)\bfe^T,
\end{equation}
where $\|\bfdb\|_2\leq \|\mathbf{U}^\dagger\|_2\|\delta
\mathbf{q}\|_2\leq \kappa_2(\mathbf{U})  \|\mathbf{b}\|_2 \bfvareps$.
Since we can express $\mathbf{e}$ as $\mathbf{e}=\mathbf{V}^T
\mathbf{f}$ with smallest possible
$\mathbf{f}=(\mathbf{V}^T)^\dagger\mathbf{e}\in\mathcal{V}$, we
obtain 
\begin{equation}
\mathbf{A}_r^\bullet = \mathbf{U}^T ( \mathbf{A} +
\bfdb\mathbf{f}^T) \mathbf{V}
 = \bfSigma_r - \mathbf{U}^T
(\mathbf{b}-\bfdb)\bfe^T.
\end{equation}
Set $\bfdA=\bfdb\mathbf{f}^T$ and note that
$\|\bfdA\|_2=\|\bfdb\|_2 \|\mathbf{f}\|_2$.
From Proposition \ref{PROP:A_r^bullet}, under the mild assumption
that $\bfvareps^\bullet < 1/2$, we conclude that
$$
\|\bfdq\|_2 \leq
\frac{2\bfvareps^\bullet}{\sqrt{r}(1-2\bfvareps^\bullet)}\|\mathbf{A}_r\|_2,\;\;
\mbox{and thus}\;\; \|\bfdb\|_2 \leq \frac{2
\|\mathbf{U}^\dagger\|_2
\bfvareps^\bullet}{\sqrt{r}(1-2\bfvareps^\bullet)}\|\mathbf{A}_r\|_2.
$$
Since $\|\mathbf{U}^\dagger\|_2 \leq \|\mathbf{V}\|_2$ and
$\|\mathbf{f}\|_2 \leq \sqrt{r}\|\mathbf{V}^\dagger\|_2$, we have
$$
\|\delta \mathbf{A}\|_2 \leq
\kappa_2(\mathbf{V})\frac{2\bfvareps^\bullet}{1-2\bfvareps^\bullet}\|\mathbf{A}_r\|_2,
\;\;\mbox{where}\;\;\|\mathbf{A}_r\| \leq
\frac{\kappa_2(\mathbf{V})}{\cos\angle(\mathcal{V,W})}\|\mathbf{A}\|_2.
$$
Further, it holds that
$$
\mathbf{V}\bfSigma_r - (\mathbf{A}+\bfdelta \mathbf{b}\mathbf{f}^T)
\mathbf{V} =\mathbf{V}\bfSigma_r - \mathbf{A} \mathbf{V} - \bfdelta
\mathbf{b}\mathbf{f}^T \mathbf{V} = \mathbf{b}\bfe^T -
\bfdb\bfe^T
=(\mathbf{b}-\bfdb)\bfe^T,
$$
and this implicitly enforces the interpolation conditions. \hfill$\Box$
\end{proof}

\begin{remark}
To claim Hermite interpolation, the only freedom left is to change
$\mathbf{c}$ into $\mathbf{c}+\bfdelta\mathbf{c}$ to guarantee that
$(\sigma_i\bfI - (\mathbf{A}^T +
\mathbf{f}\bfdb^T))^{-1}(\mathbf{c}+\bfdelta\mathbf{c})\in\mathcal{W}$
for $i=1,\ldots, r$. In other words, with some 
$r\times r$ matrix $\bfOmega$, we should have
$$
\mathbf{W}\bfOmega \bfSigma - (\mathbf{A}^T +
\mathbf{f}\bfdb^T)\mathbf{W}\bfOmega =
(\mathbf{c}+\bfdelta\mathbf{c})\bfe^T.
$$
If $\bfOmega$ commutes with $\bfSigma$, then $\bfdelta\mathbf{c}\bfe^T = \mathbf{c}\bfe^T (\bfOmega - \bfI) - \mathbf{f}\bfdb^T\mathbf{W}\bfOmega$. We can take $\bfOmega = \bfI$ and instead of the equality (which is not possible to to obtain), we can choose $\bfdelta \mathbf{c} = - (1/r)\bff\bfdb^T\mathbf{W}\bfe$,  which is the least squares approximation.
Even though this least-squares construction might provide a near-Hermite interpolation, a more elaborate construction is needed to obtain  exact Hermite interpolation for a backward  perturbed system. The framework that Beattie \emph{et al.} \cite{beattie2012inexact} provided for Hermite interpolation of a backward perturbed system in the special case of inexact solves might prove helpful in this direction.
\end{remark}

Our analysis in Section \ref{SS=Limitations-fin-prec}, specifically Theorem \ref{mumutildethm}, illustrated that the condition number of the Cauchy matrix $\bfC$ plays a crucial role in the perturbation analysis. And the numerical examples showed that despite  Cauchy matrices are known to be extremely ill-conditioned, the \textsf{IRKA} iterations drastically reduced these conditions numbers as the shifts converge to the optimal ones, i.e., as \textsf{IRKA} converges. Our analysis in this section now reveals another important quantity measure: $\frac{\kappa_2(\mathbf{V})}{\cos\angle(\mathcal{V,W})}$. Next, we will repeat the same numerical examples of Section \ref{SS=Limitations-fin-prec} and inspect how $\kappa_2(\bfV)$  and
$\frac{\kappa_2(\mathbf{V})}{\cos\angle(\mathcal{V,W})}$
vary during \textsf{IRKA}.

\begin{example}  \label{example:condV}
We use the same models and experiments from Example \ref{example:condC}. During the reduction of the CD player model to $r=2$, $r=16$, and $r=26$ via \textsf{IRKA}, we record the evolution of $\kappa_2(\bfV)$  and
$\frac{\kappa_2(\mathbf{V})}{\cos\angle(\mathcal{V,W})}$. The results depicted in Figure \ref{fig:cdcondV} show a similar story: Both quantities are drastically reduced during the iteration and thus leading to significantly smaller backward  errors $\| \bfdq\|$ and $\| \bfdelta\bfA\|$ in Theorem \ref{TM:dA_for_A^bullet}.

\begin{figure}[hh]
\begin{center}
\includegraphics[scale=0.38]{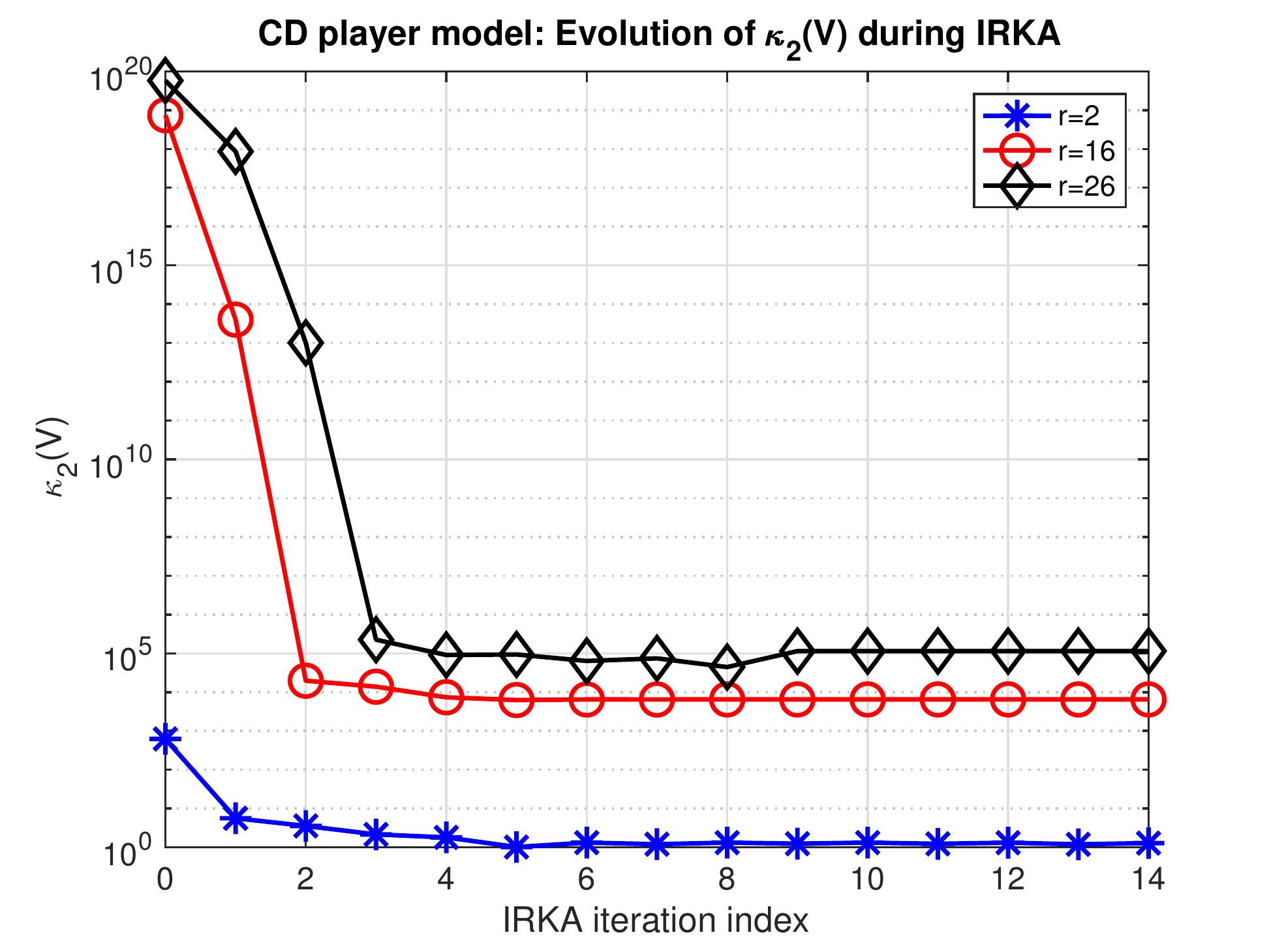}
\includegraphics[scale=0.38]{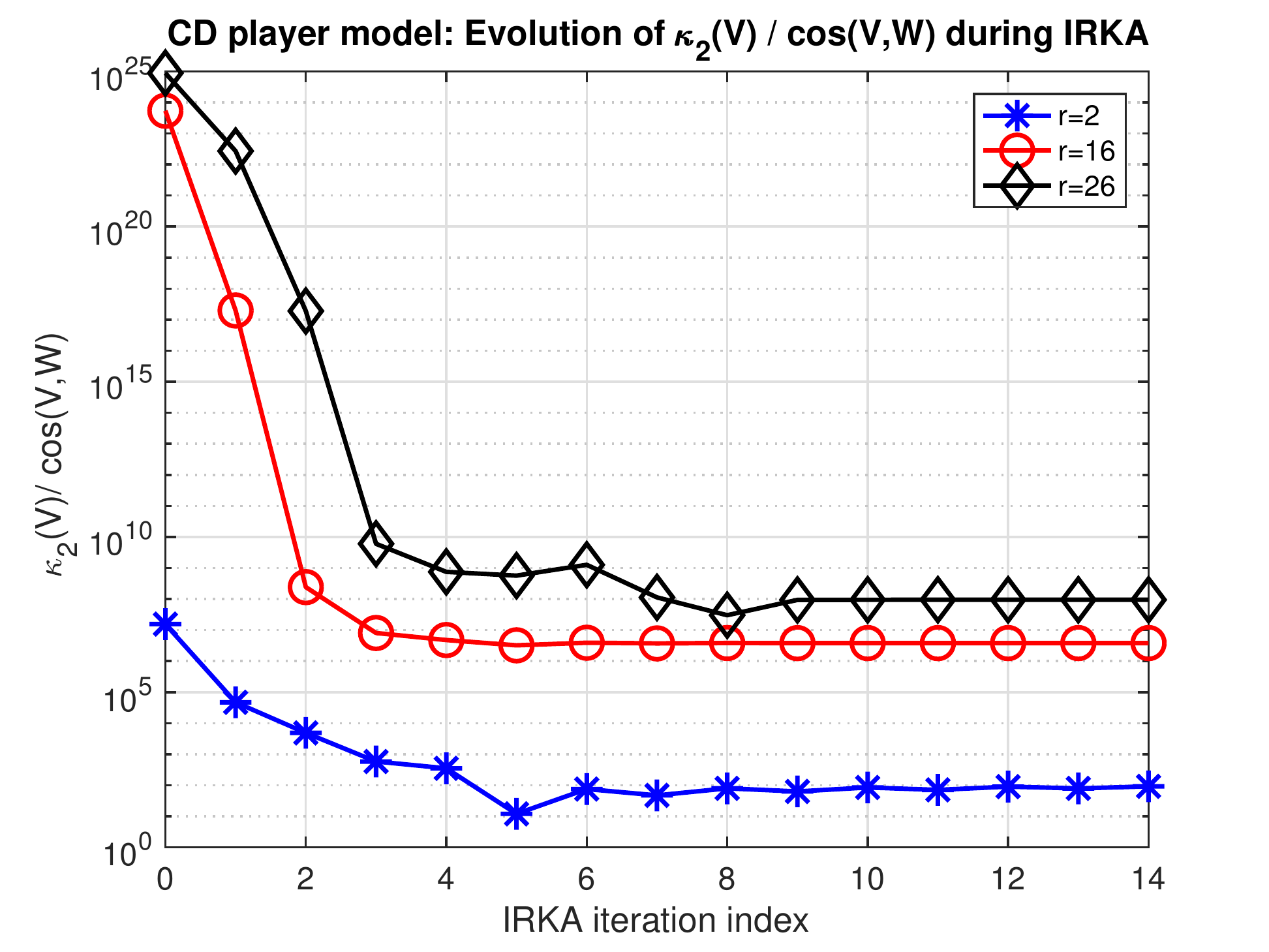}
\end{center}
\caption{$\kappa_2(\bfV)$ and $\kappa_2(\bfV)/\cos(\mathcal{V},\mathcal{W})$ during \textsf{IRKA} for the CD Player example}
\label{fig:cdcondV}
\end{figure}

We repeat the same experiments for the ISS 1R model and the results are shown in Figure \ref{fig:isscondV}. The conclusion is the same: $\kappa_2(\bfV)$  and
$\frac{\kappa_2(\mathbf{V})}{\cos\angle(\mathcal{V,W})}$ are reduced ten orders of magnitudes during \textsf{IRKA}.
\begin{figure}[hh]
\begin{center}
\includegraphics[scale=0.38]{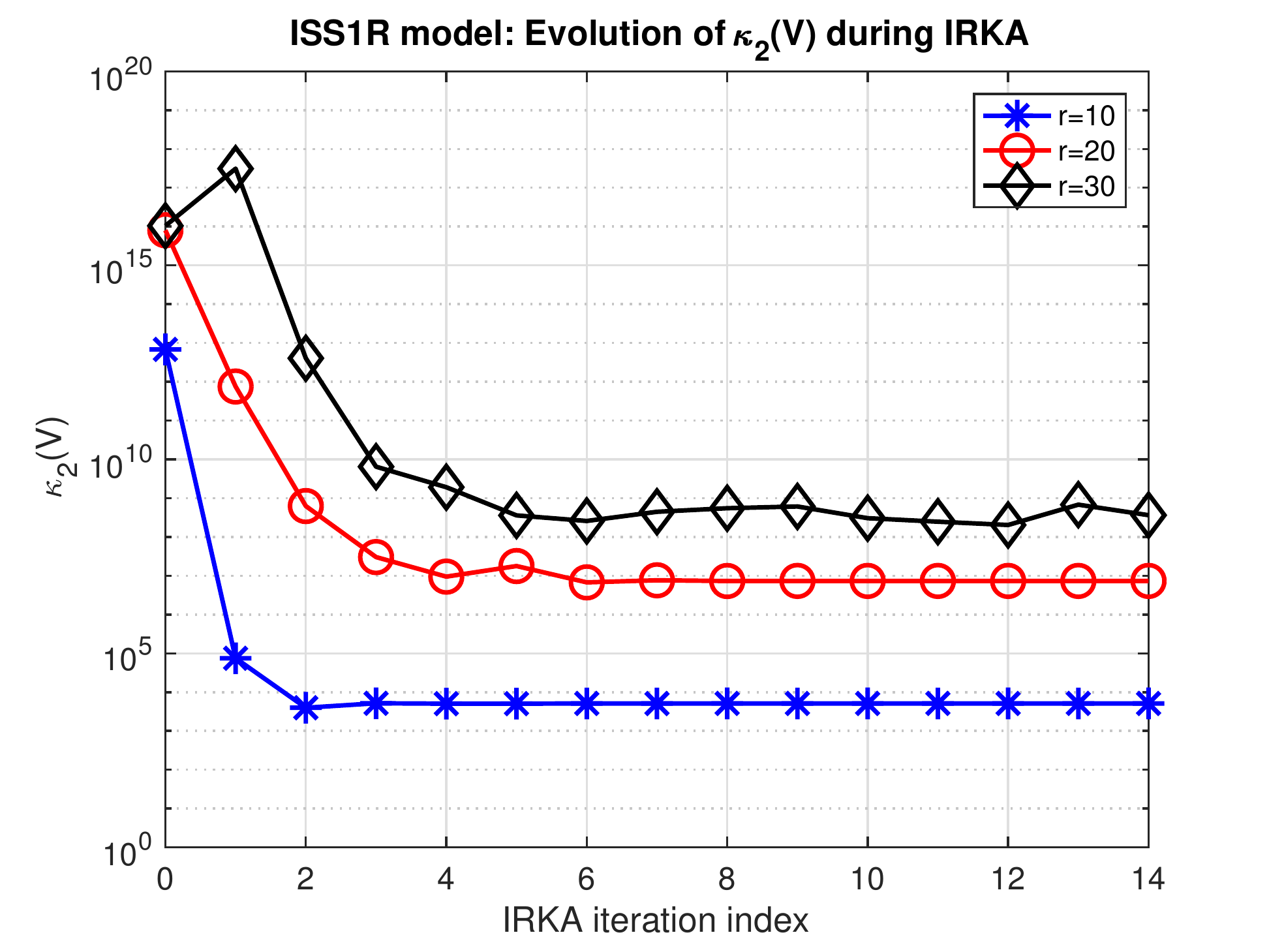}
\includegraphics[scale=0.38]{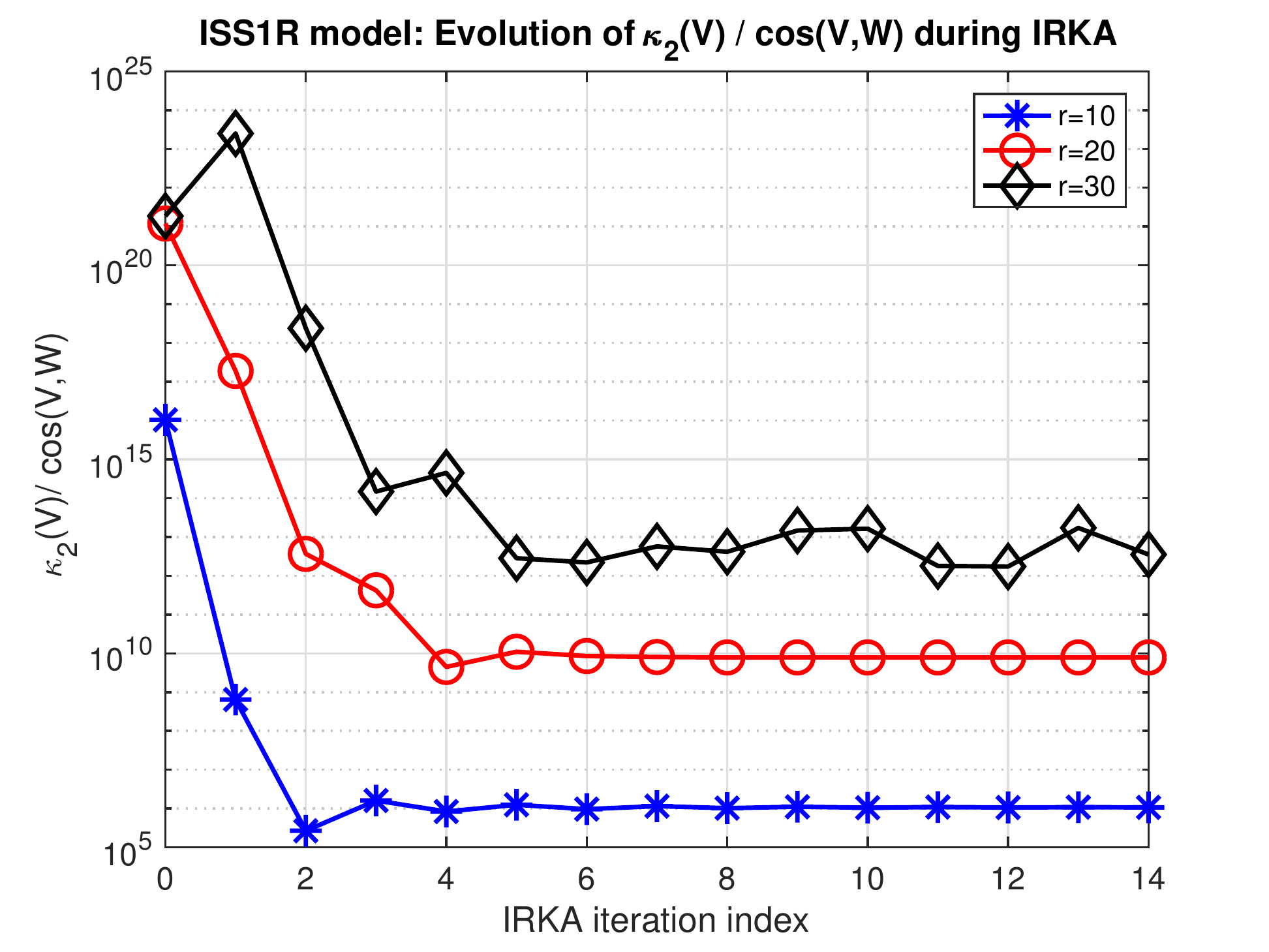}
\end{center}
\caption{$\kappa_2(\bfV)$ and $\kappa_2(\bfV)/\cos(\mathcal{V},\mathcal{W})$ during \textsf{IRKA} for ISS1R Example}
\label{fig:isscondV}
\end{figure}
\end{example}

\section{Conclusions}

By employing primitive rational Krylov bases, we have provided here an analysis for the structure of reduced order quantities appearing in \textsf{IRKA} that reveals a deep connection to the classic pole-placement problem.   We exploited this connection to motivate algorithmic modifications to \textsf{IRKA} and developed a complementary backward stability analysis. 
Several numerical examples demonstrate  \textsf{IRKA}’s remarkable tendency to realign shifts (interpolation points) in a way that drastically reduces the condition numbers of the quantities involved, thus minimizing perturbative effects and accounting in some measure for  \textsf{IRKA}’s observed robustness. 
\bibliographystyle{spmpsci}
\bibliography{ARK_bib}
 
\end{document}